\documentclass{cedram-aif}

\usepackage[english,frenchb]{babel}

\usepackage{hyperref}
% ...

\newcommand{\sub}{\mathrm{Sub}}
\newcommand{\eps}{\varepsilon}

\newcommand{\wdt}[1]{\widetilde #1}

\newcommand{\ovl}[1]{\overline #1}

\newcommand{\pl}{\partial}
\newcommand{\bs}{\backslash}

\newcommand{\vol}{\operatorname{vol}}
\newcommand{\Id}{\operatorname{Id}}

\newcommand{\tr}{\operatorname{tr}}

\newcommand{\Ima}{\operatorname{Im}}

\newcommand{\inj}{\operatorname{inj}}
\newcommand{\sys}{\operatorname{sys}}

\newcommand{\Log}{\mathrm{Log}\,}
\newcommand{\fix}{\mathrm{Fix}}
\newcommand{\Min}{\mathrm{Min}}
\newcommand{\ram}{\mathrm{Ram}}

\newcommand{\G}{\mathbf{G}}
\newcommand{\PSL}{\mathrm{PSL}}
\newcommand{\SL}{\mathrm{SL}}
\newcommand{\SU}{\mathrm{SU}}
\newcommand{\SO}{\mathrm{SO}}
\newcommand{\GL}{\mathrm{GL}}
\newcommand{\B}{\mathbf{B}}
\newcommand{\N}{\mathbf{N}}
\newcommand{\T}{\mathbf{T}}

\newcommand{\M}{\mathbf{M}}

\newcommand{\so}{\mathcal{O}}
\newcommand{\sd}{\mathcal{D}}
\newcommand{\mB}{\mathcal{B}}
\newcommand{\sF}{\mathcal{F}}

\newcommand{\frg}{\mathfrak{g}}

\newcommand{\fra}{\mathfrak{A}}
\newcommand{\frb}{\mathfrak{B}}
\newcommand{\frc}{\mathfrak{C}}
\newcommand{\fri}{\mathfrak{I}}
\newcommand{\frp}{\mathfrak{P}}

\newcommand{\stab}{\mathrm{Stab}}

\newcommand{\NN}{\mathbb N}

\newcommand{\CC}{\mathbb C}
\newcommand{\RR}{\mathbb R}
\newcommand{\ZZ}{\mathbb Z}
\newcommand{\HH}{\mathbb H}
\newcommand{\PP}{\mathbb P}
\newcommand{\QQ}{\mathbb Q}

\newcommand{\Ade}{\mathbb A}
% ...

\equalenv{lem}{lemm}

\title[Convergence des orbi-vari\'et\'es]{Sur la convergence des orbi-vari\'et\'es arithm\'etiques}

\alttitle{On the convergence of arithmetic orbifolds}

\author{\firstname{Jean} \lastname{Raimbault}}
\address{Institut de Math\'ematiques de Toulouse ; UMR5219 \\ Universit\'e de Toulouse ; CNRS \\ UPS IMT, F-31062 Toulouse Cedex 9, France}
\email{Jean.Raimbault@math.univ-toulouse.fr}

\thanks{During the writing of this paper I was supported by a research grant from the Max-Planck Gesellschaft. \\
I am grateful to Miklos Ab\'ert, Misha Belolipetsky, Nicolas Bergeron, Ian Biringer, Aurel Page, Alexander Rahm, Juan Souto and Akshay Venkatesh for valuable input to or remarks on the contents of this paper. The referees also pointed out numerous imprecisions, mistakes and typos in previous versions of this paper. 
}

\keywords{Vari\'et\'es hyperboliques arithm\'etiques, Multiplicit\'es limites, Vari\'et\'es tridimensionnelles}
\altkeywords{Arithmetic hyperbolic manifolds, Limit multiplicities, Three--dimensional manifolds}
\subjclass{22E40, 11F75, 11F72, 57M27}

\begin{document}

\begin{abstract}
  Cet article est consacr\'e \`a l'\'etude de la g\'eom\'etrie globale de certaines orbi-vari\'et\'es localement isom\'etriques \`a un produit d'espaces tridimensionnels et de plans hyperboliques. On d\'emontre que pour les peties dimensions (pour l'espace ou le plan hyperbolique, ou un produit de plans hyperboliques) certaines suites de telles orbi-vari\'et\'es non-compactes de volume fini convergent vers l'espace sym\'etrique en un sens g\'eom\'etrique pr\'ecis (``convergence de Benjamini--Schramm''). On traite aussi le cas des r\'eseaux arithm\'etiques maximaux en dimension trois dont les corps de traces sont quadratiques ou cubiques. Une des principales motivations est d'\'etudier l'asymptotique des nombres de Betti des groupes de Bianchi. 
\end{abstract}

\begin{altabstract}
  We discuss the geometry of some arithmetic orbifolds locally isometric to a product $X$ of real hyperbolic spaces $\mathbb H^m$ of dimension $m=2,3$, and  prove that certain sequences of non-compact orbifolds are convergent to $X$ in a geometric (``Benjamini--Schramm'') sense for low-dimensional cases (when $X$ is equal to $\mathbb H^2\times\mathbb H^2$ or $\mathbb H^3$). We also deal with sequences of maximal arithmetic three--dimensional hyperbolic lattices defined over a quadratic or cubic field. A motivating application is the study of Betti numbers of Bianchi groups. 
\end{altabstract}

\newtheorem{theostar}{Theorem}
\renewcommand*{\thetheostar}{\Alph{theostar}}
\newtheorem{conjstar}[theostar]{Conjecture}

\newtheorem*{corstar}{Corollary}
\newtheorem*{queststar}{Question}

\renewcommand{\thefootnote}{\alph{footnote}}

\maketitle

\section{Introduction}

Let $X$ be a Riemannian symmetric space without Euclidean or compact factors: this means that $X$ is obtained as a quotient $G_\infty/K_\infty$ of a semisimple, noncompact real Lie group $G_\infty$ by a maximal compact subgroup $K_\infty$. The space $X$ is a contractible manifold (homeomorphic to an Euclidean space) and can be endowed with the unique (up to homothety in each irreducible factor) Riemannian metric invariant under the action of $G_\infty$ by left translations. For any discrete subgroup $\Gamma$ in $G_\infty$ the quotient $M=\Gamma\bs X$ then has the structure of a Riemannian orbifold (i.e. there is a subset of codimension $\ge 2$ where the smooth structure and the metric can have singularities), and in particular various metric invariants can be associated to $\Gamma$:
\begin{itemize}
\item the Riemannian volume $\vol(M) \in ]0,+\infty]$; 
\item for each $x\in M$ the injectivity radius $\inj_x(M)$ is defined to be the largest $R$ such that the ball of radius $R$ around $x$ in $M$ is isometric to a ball in $X$; equivalently, choosing a lift $\tilde x$ of $x$ to $X$ one has
$$
\inj_x(M) = \frac 1 2 \inf_{\gamma\in\Gamma-\{1\}} d(\tilde x,\gamma \tilde x)
$$
and the global injectivity radius $\inj(M)$ is then defined as $\inf_{x\in M}\inj_x(M)$;
\item the maximal radius of $M$ is defined by $\max\inj (M) = \sup_{x\in M}\inj_x(M)$. 
\end{itemize}
All these invariants have been well--studied in the case of real hyperbolic manifolds, usually with the purpose of establishing universal constraints (or the lack thereof) for a given dimension: we refer to the introduction of \cite{Gendulphe} for a recent survey on this type of result. In this paper we are interested in the relations of the volume to the other invariants in some specific sequences of arithmetic locally symmetric orbifolds (of finite volume); this is related to the Benjamini--Schramm convergence which was studied (in \cite{7S}) by M. Ab\'ert, N. Bergeron, I. Biringer, T. Gelander, N. Nikolov, I. Samet and the author for locally symmetric orbifolds, which in turn has consequences on the Betti numbers (among other things) of these spaces.


\subsection{Geometric convergence of locally symmetric orbifolds}

\subsubsection{Benjamini--Schramm convergence to the universal cover}
\label{BS_intro}

The $R$-thin part of a Riemannian orbifold $M$ is by definition
$$
(M)_{\le R} = \{x\in M:\: \inj_x(M)\le R\}. 
$$
Its volume is a metric invariant of $M$. Fix $X$ a contractible complete Riemannian manifold; in \cite{7S} the notion of Benjamini--Schramm (BS) convergence of a sequence of finite--volume orbifold quotients $M_n$ of $X$ to $X$ was introduced by the following definition: $(M_n)$ is BS-convergent to $X$ if and only if for all $R>0$ we have
\begin{equation}
\frac{\vol (M_n)_{\le R}}{\vol M_n} \xrightarrow[n\to\infty]{} 0. 
\label{BS}
\end{equation}
In informal terms this means that `the injectivity radius of $M_n$ goes to infinity at almost every point'. One should see this notion of convergence as a middle ground between the so-called pointed Gromov--Hausdorff convergence to $X$, which asks that the maximal radius of the $M_n$ go to infinity, and the stronger statement that the global injectivity radius goes to infinity. 

Pointed Gromov--Hausdorff convergence is defined in a more general setting, and in fact any sequence of pointed Riemannian manifolds has accumulation points in this sense (the limit of course does not have to be a manifold itself), we refer to \cite[Chapter 10]{Petersen} for more information. Benjamini--Schramm convergence is also defined for more general sequences than only those satisfying \eqref{BS}, and the limits arising in this context are invariant random subgroups of the group of isometries of $X$ (see \ref{primer} below for a short introduction to this). An analogous notion was first considered in the context of regular graphs by I. Benjamini and O. Schramm in \cite{BS}, elaborated upon by M. Ab\'ert, Y. Glasner and B. Vir\'ag in \cite{AGV} and is the main tool used in \cite{7S}. In this Benjamini--Schramm topology, every sequence of lattices in $G_\infty$ is relatively compact (see \ref{primer} below). Thus, to prove \eqref{BS} for a given sequence one can argue by compactness as follows: any subsequence of $\mu_{\Gamma_n}$ has a limit, and if one can show in some way that any such limit must be equal to $X$ then \eqref{BS} must hold for the whole sequence. This line of argument was used in \cite{7S} to prove that if $G_\infty$ is simple and of real rank larger than 2 then any sequence of $X$--orbifolds must satisfy \eqref{BS}.


\subsubsection{Questions on the convergence of arithmetic orbifolds}

For this section we fix the Lie group $G_\infty$. Recall that a lattice in $G_\infty$ is a discrete subgroup such that the quotient $\Gamma \bs G_\infty$ carries a finite, $G_\infty$-invariant Borel measure. An important class of lattices in $G_\infty$ are the so-called arithmetic subgroups, which are `integral points' of $\QQ$-groups whose real points are isomorphic to $G_\infty$ up to a compact factor (see \cite[Section 10.3]{MR} for a short introduction to arithmetic groups). This work is mainly concerned with them and their convergence properties, and especially with the following question:

\begin{queststar}
Let $\Gamma_n$ be a sequence\footnote{In this paper such sequences will always satisfy the following nontriviality assumption: for distinct $n,n'$ the subgroups $\Gamma_n,\Gamma_n'$ are not conjugated in $G_\infty$.} of maximal arithmetic irreducible lattices in $G_\infty$; are the orbifolds $\Gamma_n\bs X$ Benjamini--Schramm convergent to $X$? 
\label{maxconv}
\end{queststar}

Recall that a lattice in $G_\infty$ is said to be maximal if it is not contained in a strictly larger discrete subgroup of $G_\infty$, and in the case where $G_\infty$ has more than one simple factor, a lattice it is said to be irreducible if its projection on every simple factor is dense. In the case where all simple factors of $G_\infty$ are of real rank 2 or higher an affirmative answer to the above question is provided by the much stronger result that {\it all} sequences of finite-volume irreducible $X$-orbifolds are BS-convergent to $X$ \cite[Theorem 1.5]{7S}. The irreducibility condition is needed in all cases, as is easily seen by considering the example of a sequence of maximal lattices which all contain a fixed lattice in one factor. In real rank one it is likely that the condition that the lattices be maximal (or congruence--see below) is needed; for all lattices in real hyperbolic spaces, and for some in complex hyperbolic ones there are sequences of finite covers which are not BS-convergent to the symmetric space (this is related to the failure of the congruence subgroup property in real rank one (see for example the survey in \cite[Chapter 7]{Lubotzky_Segal}), as can be seen from \cite[Theorem 1.12]{7S}). 

The arithmetic lattices in $G_\infty$ can be constructed as follows. We need a number field $F$ whose embeddings into $\RR$ we denote by $\sigma_1,\ldots,\sigma_{r_1}$, and $\sigma_{r_1+1},\ldots,\sigma_{r_2}$ are the remaining embeddings into $\CC$ up to complex conjugation (thus the degree $r=[F:\QQ]$ equals $r_1+2r_2$). We also need an algebraic group $\G$ over $F$, and we require that the group 
$$
G_\infty' := \G(F \otimes_\QQ \RR) = \prod_{j=1}^{r_1}\G^{\sigma_j}(\RR) \times \prod_{i=r_1+1}^{r_2} \G^{\sigma_i}(\CC)
$$
has a surjective map to $G_\infty$, with compact kernel. Suppose moreover that there is an $F$-embedding of $\G$ into some $\SL_m$, and let $\Gamma$ be the subgroup $\G(F)\cap\SL_m(\so_F)$ where $\so_F$ is the ring of integers of $F$. Then the image of $\Gamma$ is a lattice in $G_\infty$ by a theorem of A. Borel and Harish--Chandra. 

The construction above actually defines (without need to refer to an embedding into $\SL_m$) a commensurability class\footnote{Recall that two subgroups $\Gamma_1,\Gamma_2$ of $G_\infty$ are said to be commensurable if their intersection has finite index in both.} of lattices in $G_\infty$. This commensurability class contains infinitely many maximal lattices, and thus we see that there are two ways to generate sequences of maximal arithmetic lattices:
\begin{itemize}
\item By taking a sequence of maximal lattices inside a commensurability class;
\item By changing the group $\G$ and/or the field of definition $F$. 
\end{itemize}
We expect that in the first situation the sequence will always be BS-convergent to the universal cover. In the second situation we expect, with much less certainty however, that if the degree of the field is kept bounded the sequence will be convergent. One of the aims of this paper is to substantiate these expectations for low--dimensional examples with $\RR$-rank one factors, in particular hyperbolic three--manifolds.

An alternative to considering maximal lattices is to study sequences of so-called congruence lattices. It is more usual to consider congruence {\it subgroups} of a given arithmetic group: if $\Gamma$ stabilizes a lattice $L$ in a representation of $G_\infty$ on a real vector space, one defines the principal congruence subgroups of $\Gamma$ as the kernels $\Gamma(n)$ of the reduction maps $\Gamma\to\GL(L/nL)$ (where $n$ is a positive integer and $L/nL$ is considered as a $\ZZ/n\ZZ$-module), and a congruence subgroup of $\Gamma$ is any subgroup containing some $\Gamma(n)$ (see also \cite[Chapter 6]{Lubotzky_Segal}). In general, given a $F$-form $\G$ of $G_\infty$, the congruence lattices in $\G(F)$ are defined as the subgroups of $\G(F)$ which are equal to the intersection of $\G(F)$ with their closure in the group of points over finite ad\`eles $\G(\Ade_f)$ (if $\Gamma$ is a congruence lattice in $\G(F)$ then the congruence subgroups of $\Gamma$ as defined above are also congruence lattices in $\G(F)$). Some of these are closely related to maximal arithmetic lattices (see \cite{Rohlfs}), and we expect also that in a sequence of commensurability classes the congruence subgroups be BS-convergent to the universal cover when the degree of the field $F$ in the construction above is bounded. The case of congruence subgroups of a fixed arithmetic lattice was dealt with in \cite{7S} (actually only in the case of compact orbifolds--but the general case can be deduced from Theorem 1.11 in loc. cit. with little effort, see Proposition \ref{criter} below).


\subsubsection{Non-arithmetic lattices}
We should say a word on the arithmeticity assumption in the question above. The only semisimple Lie groups $G_\infty$ known to contain infinitely many commensurability classes of irreducible non-arithmetic lattices\footnote{The only other groups with known non-arithmetic lattices are $\mathrm{SU}(2,1)$ and $\mathrm{SU}(3,1)$, and these fall into finitely many commensurability classes (see for example \cite{DPP}).} are the groups $G_\infty=\SO(m,1)$ for $m \ge 2$ (and groups isogenic to those). The symmetric space associated to $\SO(m,1)$ is real hyperbolic space $\HH^m$ ; for each $m$ it can be seen that there are sequences of nonarithmetic minimal hyperbolic $m$-orbifolds which are far from being BS-convergent to the universal cover. For $m=2,3$ there are in fact sequences of nonconjugated maximal lattices in $\SO(m,1)$ with bounded covolume (obtained by gluing pants for $m=2$, and by Thurston's hyperbolic Dehn surgery in dimension 3). For general $m\ge 4$ this latter phenomenon is impossible because of H. C. Wang's finiteness theorem, but the variation on Gromov and Piatetski-Shapiro's construction of nonarithmetic lattices \cite{GPS} given in \cite{volume} implies (for every $m\ge 3$) that there is a $C>0$ and a sequence of pairwise noncommensurable hyperbolic $m$-manifolds $M_n$ such that for all $n$ and $x\in M_n$ we have $\inj_x(M_n)\le C$, so that if we take the manifold with minimal volume in the commensurability class of $M_n$ we obtain a sequence of maximal lattices with covolume going to infinity but injectivity radius $\le C$ at every point. The limit points in the BS-topology in this case are the IRS studied in section 13 of \cite{7S}.


\subsubsection{Unbounded maximal radius}

One can also ask whether the maximal injectivity radius is unbounded in a sequence of maximal (or congruence) arithmetic manifolds; this is obviously weaker than asking for BS-convergence of the sequence but still fails for sequences of nonarithmetic manifolds because of the same example. We ask the following questions. 

\begin{queststar}
Let $R>0$ and $m\ge 2$; is there only a finite number of maximal (or congruence) arithmetic hyperbolic $m$-orbifolds $M$ with $\max\inj M\le R$? 

More generally, without fixing the dimension, is there only a finite number of maximal (or congruence) arithmetic hyperbolic manifolds $M$ with $\max\inj M\le R$?
\end{queststar}

As a particular case of the second question one can ask whether for the sequence $M_n=\SO(n,1;\ZZ)\bs\HH^n$ we have $\max\inj(M_n)\xrightarrow[n\to+\infty]{} +\infty$ or not\footnote{This question was communicated to the author by M. Belolipetsky, who heard it from Jun-Muk Hwang who asked it out of algebro-geometric motivations.}. 


\subsection{Results on Benjamini--Schramm convergence}

\subsubsection{Arithmetic hyperbolic three--manifolds and Bianchi groups}

In the case where $G_\infty$ is a product of factors isomorphic to $\SL_2(\RR)$ or $\SL_2(\CC)$ and $\Gamma$ is an irreducible lattice in $G_\infty$ one can use the field $F$ generated by the set $\tr(\mathrm{ad}\Gamma)$ to define the commensurability class of $\Gamma$ as above (this field is independant of the choic of $\Gamma$ in the commensurability class ; it is called the {\em invariant trace field} of $\Gamma$, see \cite[3.3]{MR}). The main result in this paper is then the following. 

\begin{theostar}
Let $\Gamma_n$ be a sequence of maximal arithmetic or congruence lattices in $\SL_2(\CC)$ and $F_n$ the invariant trace field of $\Gamma_n$. Suppose that the $F_n$ are quadratic, or that they are cubic and either the $\Gamma_n$ are derived from a quaternion algebra or the size of the 2-torsion subgroup of $F_n$ is $\ll D_n^{0.24}$ (where $D_n$ is the discriminant of $F_n$). Let $M_n=\Gamma_n\bs\HH^3$. Then the oribifolds $M_n$ are Benjamini--Schramm convergent to $\HH^3$. In fact there is a $\delta>0$ such that for every $R>0$ there is a constant $C$ so that
$$
\vol (M_n)_{\le R} \le C(\vol M_n)^{1-\delta}
$$
holds for all $n$. 
\label{Main1}
\end{theostar}

Here are a few remarks about the hypothesis on the class group in the cubic case: it is conjectured that the $p$-torsion subgroup is of size $\ll D_F^\eps$ for all $p$ and $F$, with a bound depending on $p$ and the degree of $F$ (see for example \cite{EV}, which also provides some bounds in this direction). In the case of interest to us ($p=2$ and degree 3) the best currently known bound, due to Bhargava--Shankar--Taniguchi--Thorne--Tsimerman--Zhao \cite{BSTTTZ}, is $D_F^{0.2785}$, which misses what we need by about 0.04. On the other hand the mean size of the 2-torsion is bounded by a result of Bhargava \cite{Bhargava}, so the hypothesis that it is $\ll D_F^{0.24}$ is true for a subset of density one in the set of all cubic fields. 

The bound on the size of the thin part is reminescent of that obtained for congruence covers in \cite[Theorem 1.12]{7S}; however there is a big difference between the latter result and the one above, which is that in \cite{7S} the dependancy of $C$ on $R$ is made explicit (for compact orbifolds; the case of congruence covers of non-compact arithmetic three--manifolds is dealt with in \cite{torsion2}), which we do not do here (we use a non-explicit finiteness argument at some point). 

An important example which the result above covers is that of the sequence of the Bianchi groups. These were historically among the first arithmetic groups studied in relation with hyperbolic geometry, by L. Bianchi in his paper \cite{Bianchi}; they are parametrized by positive square--free integers $m$ as follows: for such $m$ let $D=m$ if $m=3\pmod 4$ or $4m$ otherwise and $F_D$ be the quadratic imaginary number field $\QQ(\sqrt{-m})$ (whose discriminant equals $-D$) and $\so_D$ its ring of integers. The Bianchi group associated to $m$ or $D$ is defined to be:
$$
\Gamma_D = \SL_2(\so_D). 
$$
Then $\Gamma_D$ is obviously a congruence lattice in $\SL_2(\CC)$ with invariant trace field $F_D$. As we will now explain the proof of Theorem \ref{Main1} rests on this special case (for which we actually can take $\delta=1/3-\eps$ for every $\eps>0$, see Theorem \ref{conv_Bianchi} below). The non-cocompact arithmetic lattices in $\SL_2(\CC)$ are all commensurable to one of the Bianchi groups \cite[Theorem 8.2.3]{MR}, and it is not very hard to deduce Theorem \ref{Main1} in the case of maximal nonuniform lattices from the special case of Bianchi groups; the congruence case then follows from this together with the statement that in a commensurability class a sequence of congruence lattices is BS-convergent to $\HH^3$. For the compact case one uses the `classical' proof of the Jacquet--Langlands correspondance given in \cite{BJ}, which relates the length spectra of compact congruence lattices in $\SL_2(\CC)$ defined over a quadratic imaginary field with that of congruence subgroups of the Bianchi groups. For the cubic case, instead of considering Bianchi orbifolds one needs to consider the irreducible noncompact quotients of $\HH^3\times\HH^2$. \\

Using the compactness argument outlined in \ref{BS_intro} we also prove the following result. 

\begin{theostar}
Let $M_n$ be a sequence of arithmetic hyperbolic three--orbifolds with fields of definition $F_n$ such that
\begin{itemize}
\item for each $n$, $F_n$ is a quadratic extension of a totally real subfield $B_n$;
\item the relative discriminants $D_{F_n/B_n}$ go to infinity; 
\item the absolute degree $[F_n:\QQ]$ is bounded.
\end{itemize}
Then $M_n$ is BS-convergent to $\HH^3$. 
\label{Main2}
\end{theostar}

The first hypothesis on $M_n$ is equivalent to the statement that $M_n$ is an arithmetic manifold of the simplest type (in the commonly used terminology), i.e. it contains immersed totally geodesic hypersurfaces, or equivalently its commensurability class is defined by a quadratic form over $B_n$ (using the isogeny from $\SO(3,1)$ to $\SL_2(\CC)$). Note that Theorem \ref{Main2} includes the case of lattices defined over imaginary quadratic fields whose discriminant goes to infinity but not that of lattices defined over cubic fields.


\subsubsection{$\SL_2$}

As explained above, the proof we give for Theorem \ref{Main1} in the quadratic case rests on the study of the special case of Bianchi groups. We will actually study the latter within a slightly larger problem which we now explain. Fix $(r_1,r_2)$ and let $\sF$ be the set of all number fields which have signature equal to $(r_1,r_2)$; we will denote by $r = r_1 + 2r_2$ the degree of those fields. For $F \in \sF$ let $\Gamma_F$ be the group $\SL_2(\so_F)$; then all $\Gamma_F$ are nonuniform arithmetic lattices in the Lie group $G_\infty=\SL_2(\RR)^{r_1}\times\SL_2(\CC)^{r_2}$, and we let $M_F$ be the finite--volume orbifold $\Gamma_F\bs X$ where $X=(\HH^2)^{r_1}\times(\HH^3)^{r_2}$ is the symmetric space associated to $G_\infty$. Then we can ask the following question. 

\begin{queststar}
Are the orbifolds $M_F$  BS-convergent to $X$ as $D_F\to+\infty$ and $F\in\sF$? 
\end{queststar}

Here and in the remainder of this paper, we use $D_F$ to denote the {\it absolute value} of the discriminant of the field $F$. As stated in Theorem \ref{Main1}, this has a positive answer for the Bianchi orbifolds (the case $(r_1,r_2)=(0,1)$). 

\begin{theostar}
For all $R>0,\eps>0$ there is a $C>0$ such that for any imaginary quadratic field $F$ we have
$$
\vol (M_F)_{\le R} \le C(\vol M_F)^{1-(1/3-\eps)}. 
$$
\label{conv_Bianchi}
\end{theostar}

We also study real quadratic fields $F$, for which we get the following result. 

\begin{theostar}
There exists a sequence $F_n$ of pairwise distinct real quadratic fields for which the following holds: for all $R>0, \eps > 0$ there is a $C>0$ such that for all $n$ we have 
$$
\vol (M_{F_n})_{\le R} \le C (\vol M_{F_n})^{1-1/10+\eps}. 
$$
\label{Hilbert_conv}
\end{theostar}

The $1/10$ is not optimal, even with our arguments. We note that a solution to Gauss' conjecture that there are infinitely many real quadratic number fields of class-number one would yield such a sequence (with a much better estimate), and actually a good enough approximation to it also does. Let us make this rigorous: if $F_n = \QQ(\sqrt{D_n})$ satisfies $h_{F_n}\ll D_{F_n}^a$ for some $0<a<1/4$ then the sequence $M_{F_n}$ BS-converges to $\HH^2\times\HH^2$, with an estimate on the volume of the thin part of the order $D_{F_n}^{1+2a+\eps}$ for all $\eps>0$; in particular an estimate like $h_{F_n}\ll\log D_{F_n}$ is as efficient as the solution to Gauss conjecture would be here. \\

Let us now briefly discuss the proof of these theorems; until the last steps we actually work with fields of arbitrary signature. Recall that in both cases the total volume of $M_F$ is of order $D_F^{3/2}$. To estimate the volume of the thin part of $M_F$ it is natural to distinguish between components according to whether they correspond to cuspidal elements (belonging to an $F$-rational parabolic subgroup) or not. The part corresponding to the cuspidal elements is sufficiently explicit, and we can relate its volume to well--known invariants of the number field $F$ (namely its discriminant, class-number and regulator) to prove that it is bounded above by a power $<1$ of the volume. The rest of the thin part corresponds to compact flats and we use a method already present in Bianchi's work to estimate the number of such flats in $M_F$. Namely, reduction theory provides us with a covering of $M_F$ by horoball regions (`Siegel domains'), one for each cusp of $M_F$. We actually need a quantitative description of the reduction theory of $\G$, for which we rely on results of S. Ohno and T. Watanabe \cite{Ohno_Watanabe}. It is then easy to count the number of geodesics of a given length through one of these horoballs, and summing over all cusps we get an estimate for the number of geodesics of a given length. Since the cusps of $M_F$ are in one-to-one correspondance with the class group of $F$, we get a sum over the latter. The upper bound we obtain is then roughly $D_F^{\frac{r-1}2} \sum_j|\frc_j|^{-r+2}$, where $\frc_j$ are representants of least possible norm for the ideal classes of $F$. It simplifies to $D_F$ for $r=2$; for $r\ge 4$ it diverges because of the cusp at infinity ($\frc=\so_F$), and for $r=3$ we can estimate it using an elementary argument about the distribution of norms of ideal classes (note that a much more precise statement holds: by a difficult result of M. Einsiedler, E. Lindenstrauss, P. Michel and A. Venkatesh \cite{ELMV} the ideal classes become equidistributed in the space of unimodular three--dimensional Euclidean lattices as the discriminant goes to infinity). 

The difficulties do not stop there, since we must then estimate the volume of the flats entering the count: these volumes are related to the relative regulators of quadratic extensions of $F$, and in general we did not manage to get good enough estimates for them. For $F$ imaginary quadratic there is no problem (this is the only case where the flats are 1-dimensional, i.e. geodesics), and for $F$ real quadratic we can use an ad hoc argument to find a sequence where we can control them. For the cubic case we do not know how to deal with this. 

Finally, note that via the Jacquet--Langlands correspondance, if we can estimate the number of maximal compact flats with given systole in $M_F$ for $F$ of signature $(r_1,r_2)=(r-2,1)$ we get as a corollary the statement that a sequence of minimal arithmetic hyperbolic 3--orbifolds defined over fields of degree $r$ is BS-convergent to $\HH^3$ (we do not need estimates for the volume of the flats for this--in fact we care only about those of dimension 1). This is how the cubic case of Theorem \ref{Main1} is proven.


\subsubsection{Other cases}

We could also deduce from our methods a few other examples of BS-convergent maximal sequences; since our interest is mostly in hyperbolic three--manifolds we will only state them briefly:
\begin{itemize}
\item Sequences of maximal arithmetic Fuchsian lattices defined over totally real fields of degree less than 3 are BS-convergent to $\HH^2$;
\item Sequences of irreducible cocompact lattices in $\SL_2(\RR\times\RR)$ defined over the fields in Theorem \ref{Hilbert_conv} are BS-convergent to $\HH^2\times\HH^2$. 
\end{itemize}
The argument of Theorem \ref{Main2} can also be applied to higher-degree CM-fields to yield results of convergence towards $(\HH^3)^r$ for $r>1$. 


\subsection{Applications}

\subsubsection{Quadratic forms}

As we mentioned the proof of Theorem \ref{conv_Bianchi} we give is a quantification of some of Bianchi's arguments in \cite{Bianchi}; one of the purposes of his work was to establish the finiteness of the number $h_F(d)$ of classes of integral quadratic forms over $F$ modulo $\SL_2(\so_F)$ with discriminant $d$. Bianchi's paper does not give an explicit estimate in terms of either $d$ or $D_F$: using our result we obtain the following estimate. 

\begin{corstar}
For every $\eps>0$ there is a $C_\eps>0$ such that for every discriminant $d\in\so_F$ and every quadratic field $F$ we have $h_F(d) \le C_\eps (d\cdot D_F)^{1+\eps}$. 

There are $\delta,C_\eps>0$ such that for every cubic field $F$ we have $h_F(d) \le C_\eps d^{3/2+\eps}\cdot D_F^{3/2-\delta}$. 
\end{corstar}


\subsubsection{Fibered arithmetic manifolds}

We can use Theorem \ref{Main1} to prove a result on the virtual fibrations of arithmetic manifolds; the proof of this was suggested by a discussion with J. Souto. We need to introduce some more terminology on hyperbolic three--manifolds to state the result. It is a seminal result of W. Thurston that for a so--called pseudo-Anosov diffeomorphism $\phi$ of a finitely triangulated surface $S$ of Euler characteristic $\chi(S)<0$, the three--manifold $M_\phi$ obtained as the mapping torus 
$$
M_\phi = S\times [0,1]/\sim, \: (\phi(x),1)\sim(x,0)
$$
has a complete hyperbolic structure of finite volume (this is proven for example in J. P. Otal's book \cite{Otal}). Amazingly, the converse is true up to finite covers: it was conjectured by Thurston, and recently proven by I. Agol \cite{Agol} (following work of D. Wise) that any complete hyperbolic three--manifold $M$ of finite volume admits a finite cover which is a fiber bundle over the circle; the image of a fiber in such a fibration under the covering map to $M$ is called a virtual fiber of $M$. Theorem 1.1 in \cite{BiringerSouto}, together with the fact that the Laplacian spectra on functions for congruence hyperbolic three--manifolds have a uniform lower bound \cite{tau}, implies that for a given $r$ there are only finitely many maximal or congruence lattices with invariant trace field of degree $r$ having a virtual fiber of a given genus: see Theorem 7.2 in loc. cit.. From our results we can recover the following special case of this. 

\begin{corstar}
For any $g\ge 2$ there are at most finitely many congruence arithmetic hyperbolic 3--manifolds defined over quadratic or cubic fields which contain a virtual fiber of genus $g$. 
\end{corstar}

\begin{proof}
Suppose the contrary, and let $S$ be a surface of genus $g$ which is a virtual fiber for infinitely many such manifolds. Then by Theorem \ref{Main1} we get that there exists a sequence of hyperbolic manifolds $M_n$ which is BS-convergent to $\HH^3$ and such that every $M_n$ is fibered over the circle with fiber $S$. But the latter fact implies that any pointed Gromov--Hausdorff accumulation point of the sequence $M_n$ has to be a so-called doubly degenerate manifold (cf. the proof of Theorem 12.8 in \cite{7S}), which contradicts the fact that $M_n$ BS-converges to $\HH^3$. 
\end{proof}

Replacing ``genus $\ge 2$'' by ``Euler characteristic $<0$'' we get a statement which deals also with noncompact manifolds and has the same proof. 


\subsubsection{Heegard genera}

A handlebody $H$ of genus $g$ is a regular neighbourhood of an embedding in $\RR^3$ of a wedge of $g$ circles. Taking two copies $H_1,H_2$ of $H$ and identifying their boundaries via an diffeomorphism of the boundary $\pl H$ (a closed surface of genus $g$) we get a closed three--manifold. The Heegard genus of a given closed three--manifold is defined to be the smallest $g$ such that $M$ can be obtained by the above construction with a handlebody of genus $g$. As in the case of fibered manifolds, bounds on the Heegard genus imply strong constraints on the geometry for hyperbolic manifolds. More precisely, D. Bachman, D. Cooper and M. White prove in \cite{BCW} that for every $g$ there is a $C(g)$ such that any hyperbolic three--manifold $M$ of Heegard genus $g$ has $\max\inj(M)\le C(g)$. Together with Theorem \ref{Main1} this implies the following result, which is a particular case of the results discussed by M. Gromov and L. Guth in \cite[Appendix A]{Gromov_Guth}. 

\begin{corstar}
Given $g>0$ there are at most finitely many congruence closed arithmetic hyperbolic three--manifolds of Heegard genus $g$ whose invariant trace-field is of degree 2 or 3. 
\end{corstar}


\subsubsection{Growth of Betti numbers}

One of the original motivations for the study of BS-convergent sequences of manifolds is that in such sequences one can relate the growth of Betti numbers to the so-called $L^2$-Betti numbers of the limit \cite[Theorem 1.15]{7S}. For hyperbolic three--manifolds, when the limit is the universal cover the $L^2$-Betti numbers vanish. Thus we get the following corollary of Theorem \ref{Main2}.

\begin{corstar} \label{Betti_lim}
Let $\Gamma_n$ be a sequence of maximal arithmetic or congruence lattices in $\SL_2(\CC)$ whose invariant trace fields are quadratic, or cubic and satisfying the conditions in Theorem \ref{Main1}. Then 
$$
\lim_{n\to+\infty} \frac{b_1(\Gamma_n)}{\vol(\Gamma_n\bs\HH^3)} = 0. 
$$
\end{corstar}

\begin{proof}
Let $\Gamma_n'$ be any sequence of torsion-free, finite index subgroups of $\Gamma_n$. Then by Proposition C in \cite{torsion1} (in the compact case it is Theorem 1.8 in \cite{7S}) we get that $b_1(\Gamma_n')=o(\vol\Gamma_n'\bs\HH^3)$. Since $b_1(\Gamma_n)\le b_1(\Gamma_n')$, it suffices to prove that there exists a constant $C$ independant of $n$ such that all $\Gamma_n$ have a torsion-free subgroup of index less than $C$. 

We will actually prove the following more general (well-known) claim: if $r,G_\infty$ are given, there is a $C$ such that for any arithmetic lattice $\Gamma$ in $G_\infty$ defined over a field $F$ of degree $r$, there is a torsion-free subgroup of $\Gamma$ of index less than $C$. To prove this we use a theorem of Minkowski (see \cite[Th\'eor\`eme 2.13]{spectre}) stating that if $\Pi$ is a subgroup of $\GL_d(\ZZ)$, the kernel of the reduction map from $\Pi$ to $\GL_d(\ZZ/3\ZZ)$ is torsion-free. Thus the claim reduces to showing that we can find a morphism from $\Gamma$ into $\GL_d(\ZZ)$ whose kernel is of uniformly bounded order, with a $d$ depending only on $G_\infty$ and $r$. This is well-known, and can be proven as follows. Let $\G$ be the algebraic goup over $F$ defining the commensurability class of $\Gamma$, and let $\ovl\G$ be its adjoint group and $\frg_F$ the $F$-Lie algebra of $\G$. By \cite[Proposition 1.2]{Borel_Prasad}, we have that the image $\ovl\Gamma$ of $\Gamma$ in the adjoint group of $G_\infty$ is contained in $\ovl\G(F)$ ; hence we get a map $\rho:\Gamma\to\GL(\frg_F)$, whose kernel contains only central elements of $\Gamma$, and hence is of order less than the number of roots of unity contained in $F$ (itself bounded by a constant depending only on $r$). By a local-global argument it is easily seen that $\rho(\Gamma)$ stabilizes an $\so_F$-lattice in $\frg_F$, hence $\rho$ induces a map $\Gamma\to\GL_d(\so_F)$ where $d=\dim_F(\frg_F)$ whose kernel has its order bounded by a constant depending only on $r$. Weil's restriction of scalars yields an embedding of $\GL_d(\so_F)$ into $\GL_{rd}(\ZZ)$, which concludes the proof of the claim.  
\end{proof}

In particular for Bianchi groups, whose covolume has the asymptotic behaviour $\vol(\Gamma_D \bs \HH^3) \asymp D^{3/2}$, we get the following limit:
$$
\lim_{D\to+\infty} \frac{b_1(\Gamma_D)}{D^{3/2}} = 0
$$
which does not seem to have been previously known. Lower bounds of the order $D\asymp (\vol M_D)^{\frac 2 3}$ are known for $b_1(\Gamma_D)$ (see \cite{Rohlfs_bianchi}), and there are computations of $H_1(\Gamma_D)$ for $D\le 1867$ in \cite{Rahm}. An interesting question would be to determine whether the limit 
$$
\lim_{D\to+\infty}\frac{\log b_1(\Gamma_D)}{\log D}
$$
exists, and to evaluate it. The existing data indicates that it could be as small as possible, i.e. one: it may well be that almost all the cuspidal cohomology of the Bianchi groups comes from base change (which gives a lower bound of the order of Rohlfs'); cf. \cite{RahmSeng} for more information. 


\subsubsection{Torsion homology}

As in \cite[Section 10]{7S} one can apply the convergence result above (Theorem \ref{Main1}) to the growth of the torsion homology of uniform arithmetic lattices in $\SL_2(\CC)$, with coefficients in nontrivial modules. Here is a sample of what can be proven: if $\Gamma_n$ is a sequence of maximal lattices with quadratic trace field and we choose for each $n$ an $\mathrm{ad}\Gamma_n$-stable lattice$L_n$ inside $\mathfrak{sl}_2(\CC)$ then we have 
\begin{equation} \label{torsion}
\lim_{n\to+\infty} \frac {\log |H_1(\Gamma_n, L_n)|} {\vol(\Gamma_n\bs\HH^3)} = \frac {13} {6\pi}. 
\end{equation}
and this generalizes without much problem to cubic lattices and to higher-dimensional local systems. 

On the other hand the problem for Bianchi groups seem more complicated. It seems that the arguments of \cite{Mueller_Pfaff} can be adapted to this setting to prove the required asymptotics of analytic torsion, but assuming that there would still be some nontrivial analysis to perform to be able to conclude that \eqref{torsion} holds. Let us mention here that the computations of A. Rahm alluded to above show that the order of torsion classes in the first homology with trivial coefficients of the Bianchi groups can be quite large. 


\subsubsection{Trace formulae and limit multiplicities}

The following `limit multiplicities' result is an immediate consequence of Theorem \ref{Main1} and \cite[Theorem 1.2]{7S}. 

\begin{corstar}
Let $\Gamma_n$ be a sequence of maximal or congruence uniform arithmetic lattices in $G_\infty=\SL_2(\CC)$ defined over quadratic or cubic fields with the conditions in Theorme \ref{Main1}; for a unitary representation $\pi$ of $G_\infty$ let $m(\pi,\Gamma_n)$ denote the multiplicity of $\pi$ in $L^2(\Gamma_n\bs G_\infty)$. Then for any regular, bounded open subset $S$ of the unitary dual of $G_\infty$ we have
$$
\lim_{n\to+\infty} \frac{\sum_{\pi\in S} m(\pi,\Gamma_n)}{\vol\Gamma_n\bs\HH^3} = \nu^{G_\infty}(S)
$$
where $\nu^{G_\infty}$ is the Plancherel measure on the unitary dual. 
\end{corstar}

Since $\SL_2(\CC)$ has no discrete series this implies that for a single representation $\pi$ we have $m(\pi,\Gamma_n)/\vol(\Gamma_n\bs\HH^3)\to 0$ (see \cite[Corollary 1.3]{7S}). In particular this result implies Corollary \ref{Betti_lim} on Betti numbers for uniform lattices. For the nonuniform case, rather than using an indirect argument as in the proof of above it would be more natural (and more susceptible of generalizing to other symmetric spaces, where the magic of hyperbolic Dehn surgery is not available) to give a proof using the trace formula. Such a proof was given in \cite{torsion2} for congruence covers of Bianchi orbifolds, but it needed an additional assumption on the asymptotic geometry, that the sum $\sum_j |\tau_j|^2$ over the cusps (see \eqref{def_shape} for the definition of $\tau$) be an $o(\vol)$. This condition is easily seen to be realized for Bianchi groups (I am grateful to Akshay Venkatesh for pointing this out to me) as can be established using an argument similar to that in Lemma \ref{cubic} below, and thus we have another proof of the Corollary. This proof is also the first step towards establishing limit multiplicities for the sequence of Bianchi groups; the missing ingredient for this more precise result is the control of intertwining operators at large eigenvalues (see \cite[Section 3.2]{torsion2}). 

There is a reverse result for the problem of limit multiplicities. In \cite{DRAFT_ABV} Mikl\'os Ab\'ert, Nicolas Bergeron and Bal\'int Vir\'ag extend the results of \cite{AGV} to the setting of IRSs in Lie groups to obtain the following result. Let $G_\infty$ is a simple Lie group ; if $Y$ is a Riemannian manifold let $\lambda_0(Y)$ denote the infimum of the spectrum of the Laplace operator on functions on $Y$. Then if $\nu$ is any IRS of $G_\infty$ supported on discrete subgroups and for $\nu$-almost all $\Lambda$ we have $\lambda_0(\Lambda \bs X) \ge \lambda_0(X)$ then $\nu$ is trivial. This implies that if a sequence of lattices $\Gamma _n \le \mathrm{PSL}_2(\CC)$ satisfies the condition sthat 
$$
\sum_{\substack{\lambda \in \sigma(\Delta_n) \\ \lambda \le 1}} \dim\ker(\Delta_n - \lambda) = o(\vol(M_n)
$$ 
(where $\Delta_n$ is the Laplacian of $M_n = \Gamma_n \bs \HH^3$) then the associated locally symmetric spaces BS-converge to $X$. In particular it follows that the Ramanujan conjecture for $\GL_2$ implies (a more general version of) our non-quantitative results. 


\subsection{Outline}

In Section \ref{recap} we recall the notion of Benjamini--Schramm convergence from \cite{7S} and make various general observations about it which were not included there. The next section \ref{geom} recalls (mostly) well-known results on the geometry of the groups $\SL_2(\so_F)$. We prove Theorem \ref{Main2} in Section \ref{simpletype}, and then Theorems \ref{conv_Bianchi} and \ref{Hilbert_conv} in Section \ref{main}. The end of the proof of Theorem \ref{Main1} is finally completed in the last section \ref{last} after recalling the description of maximal arithmetic lattices.


\section{Finite--volume orbifolds and Benjamini--Schramm convergence} \label{recap}

In this section we fix a real, noncompact, semisimple Lie group $G_\infty$ (we will also suppose for clarity that $G_\infty$ has no center), a maximal compact subgroup $K_\infty$ and let $X=G_\infty/K_\infty$ be the associated symmetric space. We will study the volume of the thin part of $X$-orbifolds of finite volume by decomposing it into cuspidal and compact components. We will also recall the notion of Benjamini--Schramm convergence from \cite{7S} and prove criteria for BS-convergence to the universal cover. These results were not in the original paper \cite{7S} since the focus there was on compact manifolds rather than general finite-volume orbifolds; this section should be seen as an appendix to this paper rather than a new development. We have however tried to be a little self-contained by recalling some important definitions from {\it loc. cit.}. 

\subsection{A recapitulation of Benjamini--Schramm convergence}
\label{primer}

As stated in the introduction, Benjamini--Schramm convergence of a sequence of $X$-orbifolds $M_n=\Gamma_n\bs X$ of finite volume towards $X$ means that $M_n$ converges almost everywhere in the Gromov--Hausdorff sense to $X$ as $n$ tends to infinity. This can be rephrased in more group--theoretic terms as follows: for any Chabauty neighbourhood\footnote{For the definition of the Chabauty topology we refer to the beginning of section 2 of \cite{7S}.} $W$ of the trivial subgroup $\Id$ of $G_\infty$ the proportion of $g\in G/\Gamma_n$ for which $g\Gamma_n g^{-1}$ intersects $W$ trivially tends to $1$. 

This allows us to define a more general notion of convergence for finite--volume $X$-orbifolds. Let $\sub_{G_\infty}$ be the compact space of closed subgroups of $G_\infty$ with its Chabauty topology; an invariant random subgroup (IRS) of $G_\infty$ is a probability measure on $\sub_{G_\infty}$ which is invariant under conjugation (this term was first coined in \cite{AGV}). Endowed with the topology of weak convergence the set ${\rm IRS}(G_\infty)$ of IRS of $G_\infty$ is compact. Moreover, there is a map from the isometry classes of finite-volume $X$-orbifolds to ${\rm IRS}(G_\infty)$ defined as follows: for any such orbifold $M$ we may choose a monodromy from  $\pi_1(M)$ to a lattice $\Gamma$ in $G_\infty$. To a lattice $\Gamma$ we may then associate the only $G_\infty$-invariant probability measure $\mu_\Gamma$ on $\sub_{G_\infty}$ which is supported on the conjugacy class of $\Gamma$; in other terms we take $\mu_\Gamma$ to be the pushforward of the $G_\infty$-invariant probability on $G_\infty/\Gamma$ by the map $g\gamma\mapsto g\Gamma g^{-1}$. This measure does not depend on the choice of the monodromy group $\Gamma$ and we say that a sequence $M_n$ is BS-convergent to an IRS $\mu$ if the sequence of IRS $\mu_{\Gamma_n}$ converges weakly to $\mu$. It is established in \cite[Lemma 3.5]{7S} that BS-convergence to the Dirac mass $\delta_{{\Id}}$ on the trivial subgroup is the same as BS-convergence to $X$ in the sense given in \eqref{BS}. This is intuitively clear: \eqref{BS} translates in group-theoretical terms to the statement that for almost all $g_n\in G/\Gamma_n$ the sequence $g_n\Gamma_n g_n^{-1}$ converges in Chabauty topology to the trivial group. 

Invariant random subgroups should be thought of as a generalization of lattices and their normal subgroups; there is a neat analogue of Borel's density theorem for IRS (\cite[Theorem 2.6]{7S}) which states that any nontrivial IRS of $G_\infty$ is supported on discrete, Zariski--dense subgroups of $G_\infty$. Recall that an isometry of $X$ is said to be elliptic if it has a fixed point in $X$, unipotent if it has a unique fixed point on the boundary at infinity and hyperbolic otherwise. Since a Zariski--dense subgroup must contain hyperbolic elements we get the following consequence:

\begin{lem}
Let $\mu\in{\rm IRS}(G_\infty)$ and suppose that there exists an open subset $U$ in $\sub_{G_\infty}$ such that any $\Lambda\in U$ does not contain an hyperbolic element of $G_\infty$ and $\mu(U)>0$. Then $\mu=\delta_{\Id}$. 
\label{cor_borel}
\end{lem}

In more informal language this means that for an IRS $\mu\not=\delta_{\Id}$ a $\mu$-random subgroup must contain an hyperbolic element.


\subsection{A criterion for BS--convergence in rank one}

If $M$ is a closed Riemannian manifold with (strictly) negative sectional curvature and $R > 0$ we define $N_R(M)$ to be equal to the number\footnote{This number is well-known to be finite: in the Hausdorff topology the set of closed geodesics of length $\le R$ is compact, and the negative curvature forbids that it has an accumulation point.} of closed geodesics of length less than $R$ in $M$. 

\begin{prop}
Let $X$ be a rank-one irreducible symmetric space and $M_n$ be a sequence of finite-volume $X$-orbifolds. Suppose that there exists a $\delta>0$ such that for all $n$ the systole $\sys(M_n)$ is larger than $\delta$. Then $M_n$ is BS--convergent to $X$ if and only if for all $R>0$ we have 
$$
\lim_{n\to+\infty} \frac{N_R(M_n)}{\vol M_n} = 0.
$$
\label{criter}
\end{prop}

This allows to deduce from \cite[Theorem 1.11]{7S} the following corollary. 

\begin{corstar}
Let $G_\infty$ be a Lie group of real rank 1 and let $\Gamma$ be an nonuniform arithmetic lattice in $G_\infty$. If $\Gamma_n$ is a sequence of congruence subgroups of $\Gamma$ then the sequence $\Gamma_n\bs X$ is BS-convergent to $X$. 
\end{corstar}

\begin{proof}[Proof of Proposition \ref{criter}]
For each $n$ fix a monodromy $\Gamma_n$ for $\pi_1(M_n)$ and let $\mu$ be any limit point of a subsequence of $\mu_{\Gamma_n}$. We will prove that $\mu$-random subgroups do not contain hyperbolic elements, from which it follows by the generalization of Borel's density theorem to IRS \cite[Theorem 2.6]{7S} (see Lemma \ref{cor_borel} above) that we must have $\mu=\delta_{\Id}$. 

Let $g_0\in G_\infty$ be any hyperbolic isometry, let $U$ be a relatively compact, open neighbourhood of $g_0$ in $G_\infty$ which contains only hyperbolic elements, and let $W_U\subset\sub_{G_\infty}$ be the set of closed subgroups of $G_\infty$ which contain at least one element in $U$; then $W_U$ is open in the Chabauty topology. We will see that for any $U$ as above we have
\begin{equation}
\lim_{n\to+\infty} \mu_{\Gamma_n}(W_U) = 0
\label{no_hyp_lim}
\end{equation}
whence it follows by a standard argument that for any $\mu$ as above we must have $\mu(W_U)=0$ for all $U$, which finally yields the claim that a $\mu$-random subgroup does not contain hyperbolic elements. 

To prove \eqref{no_hyp_lim} it suffices to show that there is a $R>0$ depending on $U$ such that $\mu_{\Gamma_n}(W_U)\ll (\vol M_n)^{-1}N_R(M_n)$, which we will now do. Let $x_0$ be the fixed point of $K_\infty$ in $X$ and 
$$
R = \sup_{h\in U} d_X(x_0,hx_0). 
$$
Now let $S_n$ be the compact subset of points in $M_n$ through which passes a closed curve of length $\le R$ which is homotopic to a closed geodesic in $M_n$. Then for all $x\in M_n \setminus S_n$ and $g$ such that $x = g\Gamma_n x_0$ we have that $(g\Gamma_n g^{-1})\cap U=\emptyset$ (for any $h\in g^{-1} \Gamma_n g$ the image in $M_n$ of the geodesic segment $[x_0,hx_0]$ is a closed curve homotopic to a closed geodesic and of length $\le R$), and it follows that 
$$
\mu_{\Gamma_n}(W_U)\le \frac{\vol S_n}{\vol M_n}.
$$
On the other hand by Lemma \ref{thin_flat} below we have that $\vol S_n\ll N_R(M_n)$, so \eqref{no_hyp_lim} follows from the hypothesis. 

The converse statement is \cite[Proposition 6.7]{7S} : we note that we do not use it in the paper. 
\end{proof}


\subsection{Estimating the volume of the thin part for products of hyperbolic spaces}
\label{est_thin}

In this subsection we restrict to  $G_\infty = \SL_2(\RR)^{r_1}\times\SL_2(\CC)^{r_2}$ (so that $X = (\HH^2)^{r_1} \times (\HH^3)^{r_2}$) and we let $\Gamma$ be an arithmetic irreducible lattice in $G_\infty$ (if $r_1+r_2 > 1$ then arithmeticity is a consequence of Margulis' arithmeticity theorem). Then either $\Gamma$ is co-compact or it is commensurable to $\SL_2(\so_F)$ for some number field $F$ of signature $(r_1,r_2)$. Here we will give a rough description of the $R$-thin part of $\Gamma\bs X$ in terms of the geometry of its compact flats and its cusps (the latter will be described in Section \ref{geom} in terms of number-theoretic quantities). 

\subsubsection{Geometry of flat manifolds}

For the contents of this section we refer the reader to Siegel's book \cite{Siegel}. If $\Lambda$ is a lattice in Euclidean space $\RR^r$ we denote by $\alpha_1(\Lambda) \le \ldots \le \alpha_r(\Lambda)$ the successive minima of $L$, by $\vol\Lambda$ the covolume of $\Lambda$ and we put 
\begin{equation}
\tau(\Lambda) = \frac{\vol\Lambda}{\alpha_1(\Lambda)^r} \asymp \frac{\alpha_2(\Lambda)\ldots\alpha_r(\Lambda)}{\alpha_1(\Lambda)^{r-1}}. 
\label{def_shape}
\end{equation}
The asymptotics on the right follow from Minkowski's second theorem and thus depend only on $r$; note that Mahler's criterion affirms that a set of unimodular lattices is bounded in $\SL_n(\RR)/\SL_n(\ZZ)$ if and only if $\tau$ is bounded on this set. For a compact Euclidean manifold $T=\Delta\bs\RR^r$ we define $\tau(T)=\tau(\Lambda)$ and also $\alpha_1(T) = \alpha_1(\Lambda)$ where $\Lambda$ is the translation subgroup of $\Delta$. 


\subsubsection{Neighbourhoods of compact flats}

If $g$ is a semisimple isometry of $X = (\HH^2)^{r_1} \times (\HH^3)^{r_2}$ there is a maximal subset in $X$ which is a union of flats on which $g$ acts as a Euclidean translation or trivially. We will denote it by $\Min(g)$, and by $\ell(g)$ the distance $d(x,gx)$ for $x\in\Min(g)$ (this is the minimal displacement of $g$, and $\Min(g)$ is the set on which it is attained---see \cite[II.6]{Bridson_Haefliger} for a discussion in a larger context). 

\begin{lem}\label{thin_flat}
There is a function $f:]0,+\infty[\to]0,+\infty[$ such that for any $R>0$ and any semisimple $g$ with $\ell(g)\le R$, the subset
$$
\{x\in X:\: d(x,gx)\le R\}
$$
is contained in the $f(R)$-neighbourhood of $\Min(g)$. 
\end{lem}

\begin{proof}
This is easily seen for a semisimple isometry of $\HH^2$ or $\HH^3$. The result on a product follows immediately (note that the isometry may be trivial in some factor, in which case $\Min(g)$ is not a flat subspace but a union of such). 
\end{proof}

\begin{lem}
If $\Gamma$ is an irreducible lattice in $G_\infty$ then $\Min(g)$ is a flat of $X$ for every semisimple $g\in\Gamma$. Moreover, if $\Gamma$ is co-compact, or if $\Gamma$ is commensurable to $\SL_2(\so_F)$ and $g$ does not fix a point in $\PP^1(F)$, then $\Min(g)$ projects to a compact subset in $\Gamma\bs X$. 
\label{max_flat}
\end{lem}

\begin{proof}
The first statement follows immediately from the fact that a non-trivial element in an irreducible lattice in $G_\infty$ cannot be trivial in any factor (since the lattice is arithmetic, hence obtained by restriction of scalars from a $F$-form $\G$ of $\SL_2$ for some number field $F$). 

If $g$ does not belong to a $F$-rational proper parabolic subgroup of $G_\infty$, then it is contained in a maximal torus $\T$ of $\G$ which is anisotropic over $F$. On the other hand, $\Min(g)$ is contained in the flat preserved by $\T$. By Godement's compactness criterion we know that 
$$
(\T(F \otimes_\QQ \RR)\cap\Gamma)\bs\T(F \otimes_\QQ \RR)
$$
is compact (see for example \cite[Proposition 10.16]{Raghunathan}) and it follows that $\Min(g)$ itself must map to a compact set in the quotient $\Gamma\bs X$. \end{proof}


\subsubsection{Cusps}

Here we suppose that $\Gamma$ is not uniform, so that it is a lattice in $G_\infty$ commensurable to $\SL_2(\so_F)$. The cusps of $\Gamma\bs X$ correspond to the points of the projective space $\PP^1(F)$. If $\xi \in \PP^1(F)$ the corresponding cusp is described as follows. Let $P$ be the Borel subgroup of $\SL_2(F)$ stabilising $\xi$ and $\Gamma_P = \Gamma \cap P$. Then the cusp associated to $\xi$ (or $P$) is $C = \Gamma_P \bs X$. 

Let $N$ be the unipotent subgroup of $P$ and $\Lambda = \Gamma \cap N$, viewed as a lattice in the Euclidean space $N$ and let $T = \Lambda \bs N$, a flat torus of dimension $r = r_1 + 2r_2$. Let $A$ be a maximal (split of rank $r_1 + r_2$) torus in $P$ with the following property: $\Gamma_P$ is generated by $\Lambda$ and $U = \Gamma \cap A$. Let $M$ be the $(r_1+r_2-1)$-dimensional subtorus of $A$ containing $U$, in other words 
\[
M = g \left\{ \left(\begin{pmatrix} a_1 & 0 \\ 0 & a_1^{-1} \end{pmatrix}, \ldots, \begin{pmatrix} a_{r_1+r_2} & 0 \\ 0 & a_{r_1+r_2}^{-1} \end{pmatrix}\right) : \prod_{i=1}^{r_1} |a_i| = 1 \right\}  g^{-1}
\]
for $g$ such that $g \infty = \xi$ and let $S$ be the compact manifold $U \bs M$. Then $C$ is topologically $]0, +\infty[ \times B$ where $B$ is a flat torus bundle with fiber $T$ over $S$, and the fibers over $\{ a \} \times S$ are isometric to $T$ scaled by a factor proportional to $a^{-1}$. The length element on $]0, +\infty[$ is $da/a$ and the metric on $S$ is constant along $a$ so the volume of the cusp truncated at $\eps$ is $\eps^{-r}\vol(T)\vol(S)$. 

We choose the parameter $a$ so that $\alpha_1(T) = 1$ and let $\eta \in \Lambda$ realise it (as an euclidean translation). Then the euclidean displacement of $\eta$ on the horosphere at height $a$ is equal to $a^{-1}$ and its dispacement in $X$ is $\asymp \log(1+a^{-1})$. We now fix $R>0$; we see that the part of $C$ where there is some unipotent element with a displacement less than $R$ is of volume $C(R) \vol(T)\vol(S)$ (where $C_R \asymp e^{rR}$, which we won't use in the sequel). To express this in terms of the conformal geometry of $T$ we use \eqref{def_shape}: we have $\vol(T) \asymp \tau(\Lambda)$. We record the conclusion of the discussion in the following lemma : let
\[
M_{\le R, \Lambda} = \{ x \in M : \exists \tilde x \text{ a lift of } x, \lambda \in \Lambda \setminus \{1\} : d(x, \lambda x) \le R \} 
\]
be the subset of $M$ where some non-trivial element of (the conjugacy class of) $\Lambda$ displaces less than $R$. 

\begin{lem}
Let $R > 0$. Then
\[
\vol(M_{\le R, \Lambda}) \ll_R \tau(\Lambda) \vol(S)
\]
for any maximal unipotent subgroup $\Lambda$ where $S$ is the flat $(r_1+r_2-1)$-torus associated to $\Lambda$ as above. 
\label{thin_cusp}
\end{lem}

We still have to estimate the volume of the chunk of the thin part of $\Gamma \bs X$ coming from non-unipotent elements of $\Gamma_P$. For this we define, for $\zeta \in \PP^1(F), \zeta \not= \xi$, a subgroup $U_\zeta = \stab_{\Gamma_P}(\zeta)$. Then $\Gamma_P$ is the disjoint union of $\Lambda$ and all $U_\zeta - \{1\}$. There is a unique maximal flat $F_\zeta$ stabilised by $U_\zeta$, foliated by hyperplanes on which it acts cocompactly with quotient $S_{\xi, \zeta}$. 

Now we fix $R > 0$ and we suppose that the injectivity radius of $S_{\xi, \zeta}$ is $\le R$ (there are only finitely many $\zeta$ with this property). By Lemma \ref{thin_flat} the part of $X$ where some non-trivial element of $U_\zeta$ displaces of less than $R$ is an $f(R)$-neighbourhood of $F_\zeta$. We do not need to estimate the volume of its image in $\Gamma \bs X$, only that of the part where there in no unipotent element which displaces less than all non-trivial elements of $U_\zeta$. Let :
\[
M_{\le R, \xi}^{\mathrm{nu}} = \{ x \in M : \text{ for all unipotent } \lambda \in \Gamma : d(\tilde x, \lambda \tilde x) > R \text{ and } \exists \gamma \in \Gamma_P \setminus \Lambda : d(\tilde x, \gamma \tilde x) \le R \}. 
\]
the subset of $M \setminus \bigcup_{\Lambda'} M_{\le R, \Lambda'}$ where some element of $\Gamma_P \setminus \Lambda$ displaces less than $R$. Projecting that to $F_\zeta$ we find a subset bounded by two parallel hyperplanes: we let $\ell(\zeta)$ be the distance between these two hyperplanes. Then the volume we want to estimate is bounded by $\ell(\zeta) \vol(S_{\xi, \zeta})$ and we get the following lemma.  

\begin{lem} \label{thin_Levi}
Let $R > 0$. Then 
$$
\vol(M_{\le R, \xi}^{\mathrm{nu}}) \ll \sum_{\substack{\zeta \in \PP^1(F), \, \zeta\not=\xi \\ \inj(S_{\xi, \zeta}) \le R}} \ell(\zeta)\vol(S_{\xi, \zeta}). 
$$
for all $\xi \in \PP^1(F)$, with a constant depending only on $R$. 
\end{lem}


\subsubsection{Conclusion}

Putting together all lemmas proven in this subsection we get the following result. 

\begin{prop}
Fix $\delta,R>0$; then for any irreducible arithmetic non-uniform lattice $\Gamma$ in $G_\infty$, such that $M = \Gamma\bs X$ has a systole $\sys(M)\ge\delta $, we have: 
\[
\vol M_{\le R} \ll \sum_{T:\alpha_1(T)\le R} \vol T + \sum_{j=1}^h \vol(S_j)\tau(\Lambda_j) + \sum_{j=1}^h \sum_{\substack{\zeta \in \PP^1(F), \, \zeta\not=\xi_j \\ \inj(S_{\xi_j, \zeta}) \le R}} \ell(\zeta)\vol (S_{\xi_j, \zeta}) + \sum_{j=1}^e q_j \vol M[\gamma_j]
\]
where the constant depends only on $R,\delta,X$. The first sum on the right is over top-dimensional compact flats $T$ of $M$, $q_1,\ldots,q_j$ are the orders of $\gamma_1,\ldots,\gamma_e$ where the $\gamma_j$ are generators for a set of representatives for the conjugacy classes of maximal finite cyclic subgroups in $\Gamma$, and $M[\gamma_j]$ is the image in $M$ of the fixed flat of $\gamma_j$. Finally, $\N_1,\ldots,\N_h$ are representatives for the $\Gamma$-conjugacy classes of maximal unipotent subgroups of $\SL_2(F)$, $\Lambda_j = \Gamma \cap \N_j(\RR)$, $\xi_j$ is the fixed point in $\PP^1(F)$ of $\N_j$ and the $S_j$ are compact quotients of the Levi subgroups of the parabolics associated with the $\Lambda_j$. 

In the case where $M$ is compact we have the same result, without the second and third term in the sum above. 
\label{vol_thin}
\end{prop}

\begin{proof}
If $g\in\Gamma$ is a semisimple element of infinite order with $\ell(g)\le R$, then by Lemmas \ref{thin_flat} and \ref{max_flat} the image in $M$ of the region of $X$ where $g$ displaces by distance less than $R$ is contained in an $f(R)$-neighbourhood of a compact flat $T$. The volume of the former is bounded by $C(R)\vol(T)$ where $C(R)$ depends only on $X$ and $R$ (for example it can be taken to be the volume of a $R$-ball in $X$). Hence the first term; the last one is obtained in the same way. The second and third terms follow immediately from Lemmas \ref{thin_cusp} and \ref{thin_Levi}. 
\end{proof}


\section{Proof of Theorem \ref{Main2}} \label{simpletype}

\subsection{Closed geodesics in arithmetic orbifolds}

Here we briefly explain how to describe the lengths of closed geodesics in hyperbolic arithmetic three--orbifolds (up to multiplication by a rational) and prove a finiteness result. We will suppose the reader familiar with the description of arithmetic lattices in $\SL_2(\CC)$ (which can be found for example in \cite{MR}). 

If $\Gamma$ is any arithmetic subgroup of $\SL_2(\CC)$ defined over a number field $F$ it is well--known that we can relate the lengths of closed geodesics in $\Gamma\bs\HH^3$ to the norms of units in quadratic extensions of $F$. More precisely: $\Gamma$ is commensurable to an arithmetic subgroup of $\G(F)$ for some quaternion algebra $A$ over $F$ and $\G=\SL_1(A)$. 

We will now state a general result on translation lengths on products of hyperbolic planes and spaces, to be re-used later (see also \cite[12.3]{MR}). We use $m$ to denote the logarithmic Mahler measure defined for $a\in\ovl\QQ$ by 
\begin{equation} \label{mahler}
m(a)=\sum_{\sigma}\max(0,\log|a^\sigma|)
\end{equation}
where the sum runs over all conjugates of $a$ in $\ovl\QQ$. Let $\gamma\in\G(F)$ have eigenvalues $\lambda^{\pm 1}$ and let $\sigma: F \to \CC$ is a real (resp. imaginary) embedding of $F$. There is a unique extension of $\sigma$ to $F(\lambda)$ such that $|\lambda^\sigma| > 1$, and the minimal displacement of $\gamma^{\sigma} \in \SL_2(\RR)$ (resp. $\SL_2(\CC)$) on $\HH^2$ (resp. $\HH^3$) is then equal to $\log|\lambda^\sigma|$. Note that the formula is also valid in case $A$ ramifies at $\sigma$ since it gives 0 in this case. So in the end we get that the double of the minimal displacement on the relevant symmetric space is equal to $\sum_{\sigma} \delta\max(0, \log|\lambda^\sigma|)$ where the sum is over all conjugates of $\lambda$ in $\CC$ and $\delta=2$ if $\lambda^\sigma \in \RR$ and $1$ otherwise. Let $r_1^\ram$ be the number of real places where $A$ ramifies and $r_1'=r_1-r_1^\ram$ and $\ell(\gamma)$ the minimal displacement of $\gamma$ on $(\HH^2)^{r_1'} \times (\HH^3)^{r_2}$. By the Cauchy-Schwarz inequality we finally obtain 
\begin{equation} \label{Mahler_length}
\frac { m(\lambda)} {\sqrt{r_2 + r_1'}} \le \ell(\gamma) \le 2m(\lambda). 
\end{equation}

Now we suppose that $F$ has exactly one complex place and $A$ ramifies at all real places. We may suppose that $F\subset\CC$ such that $\RR F = \CC$, and we will view $\Gamma$ as a subgroup of $\SL_1(A\otimes\CC)\cong\SL_2(\CC)$. Let $\tr$ be the reduced trace of $A$; any $\gamma\in\Gamma$ has integral trace $t=\tr(\gamma)\in\so_F$, and if it is moreover a semisimple element it can be diagonalized over $E=F(\sqrt d),\, d=t^2-4$, with eigenvalues $\eps^{\pm 1}$ where $\eps$ is a unit of $E$ such that $N_{E/F}(\eps)=1$. Note that this kernel is an abelian group of rank 1 since $E/F$ is inert at all infinite places but one. The proof for the following lemma follows an argument of N. Elkies \cite{MO}. 

\begin{lem} \label{nb_fini}
Fix an integer $r\ge 1$ and a real $R>0$; there exists a finite set $\{t_1, \ldots, t_m\}$ of algebraic integers (depending only on $R$ and the degree $r$) such that the following holds: for any number field $F$ of degree $[F:\QQ]\le r$, any quaternion algebra $A/F$ (split at an infinite place) and any hyperbolic $\gamma\in A^1$ with integral trace, of minimal displacement $\le R$ (on the symmetric space associated to $(A\otimes_\QQ \RR)^1$) there is $i \in \{1, \ldots, m\}$ such that $\tr\gamma = t_i$.   
\end{lem}

\begin{proof}
The set $T_R = \{u\in\ovl\ZZ^\times:\: [\QQ(u):\QQ]\le r,\, m(u)\le R\}$ is finite for all $R>0$, as the coefficients of the minimal polynomial of such an $u$ are bounded by polynomials in $\exp(m(u))$. By \eqref{Mahler_length} (in this case the lefmost term there is simply $m(\lambda)$), the eigenvalues of a $\gamma$ as in the statement must be some $\eps_i$ in $T_R$. Putting $T_R = \{\eps_1, \ldots, \eps_m\}$ and setting $t_i = \eps_i + \eps_i^{-1}$ we get the result. 
\end{proof}


\subsection{Convergence of orbifolds of the simplest type}

We prove here Theorem \ref{Main2}. The ingredients are only Lemma \ref{nb_fini} and the following result. 

\begin{lem}
Suppose that $\Gamma_n$ is a sequence of finite-covolume Kleinian groups such that for any $R>0$, for large enough $n$ any hyperbolic element in $\Gamma_n$ with minimal displacement less than $R$ has real trace. Then the sequence of orbifolds $\Gamma_n\bs\HH^3$ is BS-convergent to $\HH^3$. 
\label{conv_fuch}
\end{lem}

\begin{proof}
  To prove this it suffices to prove that any limit of a subsequence of $\mu_{\Gamma_n}$ is equal to the trivial IRS $\delta_{\Id}$. We will show that any such limit must be supported on non Zariski-dense subgroups of $\PSL_2(\CC)$, which forces it to be trivial by `Borel's density theorem' \cite[Theorem 2.6]{7S}. 

  The proof of this claim is similar to that of Proposition \ref{criter}: if $U$ is an open, relatively compact subset of $(\CC\setminus\RR)/\{\pm 1\}$ let $W_U$ be the open subset of $\sub_G$ ($G=\PSL_2(\CC)$) of subgroups which contain an element having its trace in $U$. We can choose a countable set $\mathcal U$ of such $U$s so that $\bigcup_{U \in \mathcal U} U = (\CC \setminus \RR)/\{\pm 1\}$. Then the complement of $\bigcup_{U \in \mathcal U} W_U$ in $\sub_G$ contains only non-Zariski-dense subgroups. Indeed, the condition $\pm\tr(g) = \pm\overline{\tr(g)}$ describes a Zariski-closed subset in $\PSL_2(\CC)$, so every Zariski-dense subgroup of $G$ must contain an element with trace in $(\CC \setminus \RR)/\{\pm 1\}$ and so belongs to $W_U$ for some $U \in \mathcal U$. 

  Thus it suffices to prove that for every $U \in \mathcal U$ we have: 
  \[
  \lim_{n\to+\infty} \mu_{\Gamma_n}(W_U) = 0
  \]
  (from which the result follows by the $\sigma$-additivity of any limit measure). But since for any given $U$, there is a $R>0$ so that $W_U$ is contained in $\{z: |z|\le R\}$ we have in fact that $\mu_{\Gamma_n}(W_U)=0$ for large enough $n$. 
\end{proof}

We now check the condition of Lemma \ref{conv_fuch} for the sequence $\Gamma_n$ in the statement of Theorem \ref{Main2}. Fix $R>0$, and let $t_1,\ldots,t_m$ be the finite set given by Lemma \ref{nb_fini}, ordered so that the totally real ones are exactly $t_1, \ldots, t_l$. Then for any $l < i \le m$ we have $t_i \not\in F_n$ for $n$ large enough: indeed, since $t_i$ is not totally real we have $t_i\in F_n \Rightarrow F_n = B_n(t_i)$ but the relative discriminants $D_{B_n(t_i)/B_n}$ for $i = 1, \ldots, m$ are bounded so we must have $t_i\not\in F_n$ for large $n$ since $D_{F_n/B_n}$ is by hypothesis unbounded. 

Thus, for $n$ large enough, the traces of the elements in $\Gamma_n$ of displacement less than $R$ are among those of the $t_i$ which are totally real numbers. 


\section{Geometry of $M_F$} \label{geom}

Recall from the introduction that if $F$ is a number field with signature $(r_1,r_2)$ we denote by $M_F$ the finite-volume orbifold 
$$
M_F = \PSL_2(\so_F)\bs(\HH^2)^{r_1}\times(\HH^3)^{r_2}. 
$$
This section is preliminary to the proof of our main results; we record various known results about the global geometry of the orbifolds $M_F$. We will also set notation as we go along. We begin by recalling the volume formula for $M_F$ given in \cite{Borel_volumes}:
\begin{equation}
\vol M_F = \frac{2\zeta_F(2)}{2^{3r_2} \pi^{r_1+2r_2}} D_F^{\frac 3 2}.
\label{vol_total}
\end{equation}

\subsection{Reduction theory for $\SL_2$}

We describe here the well--known reduction theory for the groups $\Gamma_F = \SL_2(\so_F)$ in a manner suitable to the use we will make of it later. We will work with the adelic version of reduction theory, which is much better suited to the study of manifolds which have more than one cusp. For this subsection we fix a number field $F$ with $r_1$ real places and $r_2$ complex ones, and let $r=r_1+2r_2$ be its absolute degree. We will denote by $V_f,V_\infty$ the set of finite and infinite places of $F$. For each $v$ either finite of infinite we denote by $|\cdot|_v$ the absolute value at $v$ and by $F_v$ the metric completion of $F$ for $|\cdot|_v$. We will use $\Ade_f$ to signify the ring of finite ad\`eles of $F$, and denote $F_\infty=\prod_{v\in V_\infty} F_v$ and $\Ade=F_\infty\times\Ade_f$ the complete ring of ad\`eles. If $v$ is a finite place of $F$ we let $K_v = \SL_2(\so_v)$, and we denote by $K_f$ the closure of $\Gamma_F$ in $\SL_2(\Ade_f)$ (which equals $\prod_{v\in V_f} K_v$). For each $v\in V_\infty$ we choose the maximal compact subgroup $K_v$ of $\SL_2(F_v)$ to be $\SO(2)$ if $F_v=\RR$ and $\SU(2)$ if $F_v\cong\CC$. The product $K = K_f\times\prod_{v\in V_\infty} K_v$ is then a compact subgroup of $\SL_2(\Ade)$. 

\subsubsection{Minkowski--Hermite reduction}

Let $\B$ be the standard Borel subgroup (upper triangular matrices) over $F$ for $\SL_2$, $\N$ its unipotent subgroup and $\T$ the maximal $F$-split torus of diagonal matrices. Let $h_F$ be the class-number of $F$ and $\frc_1=\so_F,\frc_2,\ldots,\frc_{h_F}$ a set of representatives for the class-group of $F$. For each $i=1,\ldots,h_F$ we choose $a_i$ to be an id\`ele of (id\'elic) norm $1$ such that for all $v\in V_f$ we have $a_i\so_v=\frc_i\so_v$; we will (abusively) also denote by $a_i$ the matrix $\begin{pmatrix} a_i & 0 \\ 0 & a_i^{-1}\end{pmatrix}\in\SL_2(\Ade)$. Let $T_\infty=\T(F_\infty)$ and for any $c>1$ define
$$
T_\infty(c) = \left\{ \begin{pmatrix} t & 0 \\ 0 & t^{-1}\end{pmatrix}\in T_\infty:\:  |t|_\infty \ge c \right\},
$$
where we put $|t|_\infty=\prod_{v\in V_\infty} |t_v|_v$. The following result is then a straighforward consequence of classical reduction theory (see for example \cite{Godement_bourbaki}): for small enough $c_F$ we have\footnote{To make the sets on the right-hand side of finite volume one usually replace $\N(\Ade)$ by a compact subset $U$ such that $\N(\Ade)=\N(F)U$; we will not need to do so here.}: 
\begin{equation}
\SL_2(\Ade) = \bigcup_{i=1}^{h_F} \SL_2(F)\cdot \N(\Ade) T_\infty(c_F) a_i K.
\label{reduction0}
\end{equation}
There is an estimate for $c_F$ for number fields due to S. Ohno and T. Watanabe \cite{Ohno_Watanabe}: they prove that 
\begin{equation}
\gamma(F):= \sup_{g\in\GL_2(\Ade)} \inf_{v\in F^2} \left( \frac{\| gv\|^{\frac 2 r}}{|\det g|^{\frac 1 r}} \right) \le C D_F^{\frac 1 r}
\label{boundheight}
\end{equation}
where $\|\cdot\|$ is a norm on $\Ade^2$ given at each place $v$ by a norm preserved by $K_v$ and normalized so that $\|e_1\|=1$. It is easy to see that we can take $c_F\ge\gamma(F)^{-\frac r 2}$ in \eqref{reduction0}. 

Using strong approximation for $\G$ it is possible to replace each $a_i$ in \eqref{reduction0} by an element of $\SL_2(\Ade)$ which is equal to some $\gamma_i\in\SL_2(F)$ at all finite places. We will give an explicit choice for these $\gamma_i$ (which will actually circumvent the use of the strong approximation property of $\G$). For $\xi\in F$ we put
\begin{equation}
\gamma_\xi = \begin{pmatrix} 0&1\\-1&\xi\end{pmatrix}.
\label{gammaxi}
\end{equation}
Let $v\in V_f$; if $v(\xi)\le 1$ then $\gamma_\xi\in K_v$; if $v(\xi)>1$ then one can easily compute the Iwasawa decomposition of $\gamma_\xi$ at $v$:
\begin{equation}
\gamma_\xi =  \begin{pmatrix} \xi^{-1} & 1 \\ 0 & \xi\end{pmatrix} \begin{pmatrix}  1 & 0 \\-\xi^{-1} & 1 \end{pmatrix}.
\label{iwasawa}
\end{equation}
Fix $i=1,\ldots,h_F$; choose $\alpha,\beta\in\so_F$ such that\footnote{This is possible since a=every ideal of $\so_F$ is generated by two elements, hence if we take $\alpha,\beta$ to generate a representative for the inverse class of $\frc_i$ we get that the integral ideal $(\beta)/(\alpha,\beta)$ is a representative of $\frc_i$.} $\frc_i=(\beta)/(\alpha,\beta)$ and let $\xi=\alpha/\beta$ then it follows from \eqref{iwasawa} that
$$
\N(\Ade_f)a_iK_f=\N(\Ade_f)\gamma_\xi K_f. 
$$
Putting $\gamma_i=\gamma_\xi$ we thus get from \eqref{reduction0}
\begin{align*}
\SL_2(\Ade) &= \bigcup_{i=1}^{h_F} \SL_2(F)\cdot \N(\Ade) T_\infty(c_F) (a_i)_\infty K_\infty (a_i)_fK_f \\
        &= \bigcup_{i=1}^{h_F} \SL_2(F)\cdot \N(\Ade) T_\infty(|a_i|_\infty c_F)K_\infty \gamma_i K_f,
\end{align*}
as we have $|a_i|_\infty=|\frc_i|$ we can write the last equality as:
\begin{equation}
\SL_2(\Ade) = \bigcup_{i=1}^{h_F} \SL_2(F) \N(\Ade) T_\infty(c_i) K_\infty (\gamma_i K_f)
\label{reduction1}
\end{equation}
where $c_i=|\frc_i|\cdot c_F$. 


\subsubsection{Height functions}
\label{height}

Let $\alpha$ be the simple root of $\T$ defined over $F$ given by 
$$
\alpha \begin{pmatrix} t & 0 \\ 0 & t^{-1} \end{pmatrix}  = t^2. 
$$ 
We extend it as a character of $\B$ trivial on $\N$, and we define a function on $\SL_2(\Ade)$ by $\alpha(bk) = \alpha(b)$. Then the height function $y_\Ade$ on $\SL_2(\Ade)$ is defined by $y_\Ade(g)=\max_{\gamma\in\SL_2(F)}|\alpha(\gamma g)|$. We will identify this adelic function with height functions on the symmetric space $X$ in the sense of \cite[2.1,2.2]{torsion1}. 

First we need to fix once and for all a model for $X$: for $v\in V_\infty$ we identify $X_v=\SL_2(F_v)/K_v$ with $F_v\times]0,+\infty[$ using the Iwasawa decomposition in the usual manner where $(0,1)$ is the fixed point of $K_v$, so that $X=G_\infty/K_\infty=\prod_{v\in V_\infty}$ is identified with a product $F_\infty\times]0,+\infty[^r$ of half-planes and -spaces. We define a height function $y$ on $X$ as follows: for $(z,y)\in F_\infty\times]0,+\infty[^r$ put $y_\infty(z,y) = \prod_{v\in V_\infty} |y_v|_v$, and for $\xi\in \PP^1(F)$ choose a $\gamma\in\SL_2(F)$ such that $\gamma\cdot\xi=\infty$ and an Iwasawa decomposition $\gamma=b_f k_f$ with $k_f\in K_f,b_f\in\B(\Ade_f)$ (see \cite[Proposition 4.5.2]{Bump}), and put
\begin{equation}
y_\xi(x) = |\alpha(b_f)|_f  \cdot y_\infty(\gamma x)
\label{def_height}
\end{equation}
(which does not depend on the particular $\gamma$ chosen). Then $y(x)=\max_{\xi\in\PP^1(F)} y_\xi(x)$ is a $\Gamma_F$-invariant height function, and if $\pi$ is the natural isomorphism $\SL_2(F)\bs\SL_2(\Ade)/K\to \Gamma_D\bs\HH^3$ we have $y_\Ade=y\circ\pi$. 

We record that if $\xi\in F$ and $\gamma_\xi$ is defined as in \eqref{gammaxi} then \eqref{iwasawa} yields:
\begin{equation}
y_\xi(x) = |\frb|^{-1} y_\infty(\gamma_\xi x)
\label{calcy}
\end{equation}
where $\frb$ is the ideal $(\beta)/(\alpha,\beta)$ if $\xi=\alpha/\beta$. 


\subsubsection{A fundamental set in classical setting}

From the decomposition \eqref{reduction1} one can deduce a fundamental set for $\Gamma_F$ in $X$: let $x_0$ be the fixed point of $K_\infty$ and for $Y>0$ let 
$$
B_\infty(Y) = \{x \in X :\: y_\infty(x) \ge Y\}
$$
be the horoball about infinity of height $Y$; one has $B_\infty(Y)=T_\infty(Y)\N(F_\infty)\cdot x_0$. Thus it follows from \eqref{reduction1} that 
\begin{equation}
X = \Gamma_F \bigcup_{i=1}^{h_F} \gamma_i^{-1} B_\infty(Y_i).
\label{reduction2}
\end{equation}
where for each $i=1,\ldots,h_F$ we put $Y_i=c_i=c_F|\frc_i|$. 


\subsection{Cusps} \label{cusps}

In this subsection we make the discussion of cusps in section \ref{est_thin} more explicit. 

\subsubsection{Cusp stabilizers}

Let $\xi \in \PP^1(F) = F \cup \{\infty\}$ and let $P_\xi$ its stabiliser in $G_\infty = \SL_2(\RR)^{r_1} \times \SL_2(\CC)^{r_2}$. Let $\Gamma_\xi = \Gamma \cap P_\xi$; if $\xi =\infty$ then we have 
\begin{equation} \label{stab_infty}
\Gamma_\infty = \left\{ \begin{pmatrix} t & z\\ & t^{-1} \end{pmatrix}:\: t\in\so_F^\times, z\in\so_F \right\}. 
\end{equation}
For $\xi=\frac\alpha\beta \in F$ let as before $\frb$ be the ideal $(\beta)/(\alpha,\beta)$. Then an easy computation yields
\begin{equation} \label{stab_xi}
\Gamma_\xi = \left\{ \begin{pmatrix} t^{-1} - z\xi & z\xi^2 + (t - t^{-1})\xi \\ -z & t + z\xi \end{pmatrix} :\: t \in \so_F^\times, z \in \frb, z\xi + (t - t^{-1}) \in \frb \right\}. 
\end{equation}


\subsubsection{Unipotent part}

Let $N_\xi$ be the unipotent radical of $P_\xi$: in this subsection we study the shape of the flat tori $T_\xi = (N_\xi \cap \Gamma) \bs N_\xi$. The unipotent subgroup of $\Gamma_\infty$ is $1 + \so_F X$ where $X = \begin{pmatrix} 0&1\\0&0 \end{pmatrix}$. It follows that $T_\infty$ is isometric to $\so_F \bs F_\infty$. From \eqref{stab_xi} we compute that the unipotent part of $\Gamma_\xi$ is:
\begin{equation}
\Lambda_\xi = \left\{ 1+\begin{pmatrix} -z\xi & z\xi^2 \\ -z & z\xi \end{pmatrix}:\: z\in\frb^2 \right\}
\label{unip_xi}
\end{equation}
and it follows that $T_\xi$ is isometric to $(\frb^2\bs F_\infty)$.

There is a crude estimate for the shape of the lattice in $F_\infty$ associated to an ideal of $\so_F$. Let $\fra$ be an ideal and $a \in \fra$ such that $\| a \| = \alpha_1(\fra)$. Then $\fra$ contains the finite-index lattice $a\so_F$ which has the same $\alpha_1$, and it follows that $\tau(\fra) \le \tau(a\so_F) = \tau (\so_F) = D_F^{1/2}$. So we get: 
\begin{equation}
\tau(T_\xi) = \tau(\frb^2) \le D_F^{1/2}. 
\label{shape}
\end{equation}

We will also need (because of the term $\ell(u)$ in Proposition \ref{vol_thin}) an estimate of the displacement of elements in $\Lambda_\xi$ at a given height. It follows easily from the explicit description \eqref{unip_xi} and Minkowski's first theorem implies that on the horosphere $H = \{y_\xi = 1\}$, we have $\alpha_1(\Lambda_\xi\bs H) \ll |\frb|$ with an absolute constant. 


\subsubsection{Levi component}

Here we describe the geometry of the flat tori $S_\xi$ (which may be called the ``Levi components'' of the cusps, since they come from the Levi components of the $P_\xi$). We will use the usual notation $\Log$ for the map $\so_F^\times\to\RR^{r_1 + r_2 - 1}$ defined by $u \mapsto (\log |u|_v)_{v\in V_\infty}$, where we have fixed an isometry $\RR^{r_1+r_2-1}\cong\{(x_v)_{v\in V_\infty} \in \RR^{V_\infty} :\: \sum_v x_v=0\}$, and $R_F$ for the regulator, i.e. the covolume of $\Log\so_F^\times$. 

We will work at all places of $F$ and hence need to describe things using $F$-algebraic groups. We fix $\xi \in \PP^1(F)$ and let $\B_\xi$ be the stabiliser of $\xi$ in $\SL_2/F$, and $\N_\xi$ its unipotent radical. Recall that for any finite place $v$ of $F$ we have defined $K_v = \SL_2(\so_v)$. In the sequel we will make use of the Bruhat-Tits tree of $\SL_2(F_v)$ which we will denote by $X_v$. 

\begin{lem} \label{good_torus}
There exists an $F$-split torus $\T_\xi$ in $\B_\xi$ such that for every finite place $v$, $\T_\xi(F_v) \cap K_v$ is a maximal compact subgroup in $\T_\xi(F_v)$.  
\end{lem}

\begin{proof}
Let $x_v$ be the fixed point of $K_v$ in $X_v$. It suffices to show that there exists a $\zeta\in F$ such that for each $v\in V_f$ the geodesic $]\xi,\zeta[$ determined by $\xi,\zeta$ in $X_v$ goes through $x_v$; starting from any $\zeta_0\in F$ we can move it at each $v$ by an $n_v\in\N(F_v)$ so that $x_v\in]\xi,n_v\zeta_0[$; strong approximation for $\N$ tells us that there is a $n\in\N(F)$ such that $n_v\N(\so_v)=n\N(\so_v)$ for all $v$, and since $\N(\so_v)$ fixes $x_v$ it follows that $\zeta=n\zeta_0$ yields what we need.
\end{proof}

We have an identification of $\T_\xi(F) \cap \Gamma_\xi$ with the group of units $\so_F^\times$, so that $S_{\xi, \zeta_0}$ is isometric to $\Log(\so_F^\times) \bs \RR^{r_1+r_2-1}$. The following lemma gives a similar description for the other $S_{\xi, \zeta}$, where we denote by $\so_F^\times(\fri)$ the congruence subgroup $(1 + \fri\so_F) \cap \so_F^\times$. 

\begin{lem} \label{Levi_local}
Let $\zeta \in \PP^1(F), \zeta\not= \xi$, let $x_v' \in X_v$ such that $]\zeta, \xi[ \cap ]\zeta_0, \xi[ = [x_v', \xi[$, and $m_v = d_{X_v}(x_v, x_v')$ in case $x_v$ lies on $[x_v', \xi[$ and 0 otherwise. Put $\fri_{\xi, \zeta} = \prod_{v\in V_f} \frp^{m_v}$; then $S_{\xi, \zeta}$ is isometric to $\Log(\so_F^\times(\fri)) \bs \RR^{r_1+r_2-1}$.
\end{lem}

\begin{proof}
We denote by $\T_{\xi, \zeta}$ the $F$-torus in $\SL_2$ fixing $\xi$ and $\zeta$. Let $U_v = K_v \cap \T_{\xi, \zeta}(F_v)$ and $U_\zeta = T_{\xi, \zeta}(F ) \cap \left( \prod_{v\in V_f} U_v \right)$. Let $\alpha$ be any nontrivial $F$-character of $\T_{\xi,\zeta}$ and let $\M_{\xi, \zeta}$ be the $\QQ$-subgroup $\ker(N_{F/\QQ} \circ \alpha)$. Then $U_\zeta \subset \M_{\xi,\zeta}(F)$ and we have $M := \M_{\xi, \zeta}(\RR) \cong \RR^{r_1+r_2-1} \times \mathrm U(1)^{r_2}$. We get that $S_{\xi, \zeta}$ is isometric to $U_\zeta \bs M / \mathrm U(1)^{r_2}$, so we need to describe $U_\zeta$.  

In case $x_v$ lies on $[x_v', \xi[$ we have that $U_v$ is a maximal compact subgroup. Now suppose that $m_v>0$. We prove that for every $v \in V_f$ we have that $U_v = \T_{\xi, \zeta}(F_v) \cap K_v$ is a subgroup identified with $1+\frp_v^{m_v}\so_v$ in a maximal compact subgroup of $\T_{\xi, \zeta}(F_v)$ identified with $\so_v^\times$. We have that $U_v$ fixes $x_v'$ since this point is on $[x_v, \xi[$ and $U_v$ fixes both $x_v$ and $\xi$. Thus we see that $U_v$ is the subgroup of the maximal compact subgroup $\stab_{\T_{\xi, \zeta}(F_v)}(x_v')$ acting trivially on the ball of radius $m_v$, which is exactly the right subgroup. By strong approximation for $\SL_2$ we can conclude that $U_\zeta \subset \T_{\xi, \zeta}(F)$ is conjugated by an element of $\SL_2(F)$ to 
\begin{equation} \label{Levi_inc}
\left\{ \begin{pmatrix} u & 0 \\ 0& u^{-1} \end{pmatrix} :\: u \in \so_F^\times(\fri) \right\} \subset \left\{ \begin{pmatrix} t & 0 \\ 0& t^{-1} \end{pmatrix} :\: t \in F^\times \right\}. 
\end{equation}
It follows from \eqref{Levi_inc} that $S_{\xi, \zeta}$ is isometric to $(\Log \so_F^\times(\fri))\bs\RR^{r_1+r_2-1}$. 
\end{proof}

It follows from Lemma \ref{Levi_local} that
\begin{equation}
\vol S_{\xi, \zeta} = R_F\cdot\phi(\fri) 
\label{vol_Levi}
\end{equation}
where $\phi(\fri)$ is the number of units of the ring $\so_F/\fri$, so that $\phi(\fri)\le|\fri|$. We also get that: 
\begin{equation}
\inj S_{\xi, \zeta} \gg\log|\fri| 
\label{sys_Levi}
\end{equation} 
with a constant depending only on $r$. Indeed, the coefficients of the elements in $\Log\so_F^\times(\fri)$ in any basis of $\Log(\so_F^\times)$ are $\ge\log|\fri|$, and the minimal norm of an element of $\log\so_F^\times$ is bounded below by a constant depending only on $r$.

Finally, we can also describe $\ell(\zeta)$ in this setting: let $\zeta \in \PP^1(F), \zeta \not= \xi$, then on the geodesic between them in $X$ we have
\begin{equation} 
|\log y_\xi - \log y_\zeta| \le \log|\fri_{\xi,\zeta}| + \frac 1 2\log D_F. 
\label{comp_height}
\end{equation}
To see this we argue as follows: let $\T_{\xi,\zeta}$ as above and choose $\gamma_0\in\SL_2(F)$ sending $(0,\infty)$ to $(\zeta, \xi)$. For each $v\in V_f$ let $x_v'$ be the projection on the geodesic $(\zeta,\xi)$ in $X_v$ of the point $x_v$ and $x_v''=\gamma_0 x_v$. We can choose for each $v$ an element $a_v\in\T_{\xi,\zeta}(F_v)$ such that $a_vx_v''=x_v'$, and then $(a_v)_{v\in V_f}$ can be approximated, up to an element $\frc$ of the id\`ele class-group, by an element $a_0\in\T(F)$. Let $\gamma_1=a_0\gamma_0$, $\gamma_1'=\gamma_1 w$ (where as usual $w=\begin{pmatrix} 0&-1\\1&0\end{pmatrix}$), so that we have: 
\begin{itemize}
\item[(i)] $\gamma_1\cdot\infty= \xi$ and $\gamma_1'\cdot\infty=\zeta$;
\item[(ii)] For all $v\in V_f$, $d_{X_v}(x_v,\gamma_1^{-1} x_v)=d_{X_v}(x_v,(\gamma_1')^{-1} x_v)\le \log_{q_v}|\frc_v| + m_v$. 
\end{itemize}
The claim that \eqref{comp_height} holds then follows from (i) and (ii) and the definition \eqref{def_height}, as we have $|\frc|\le D_F^{1/2}$, $|\log\alpha(b_v)| \le \log d_{X_v}(x_v,b_v x_v)$ and $\alpha(wbw^{-1})=\alpha(b)^{-1}$ for all $b\in\B(F_v)$ (and recall that $\fri_{\xi, \zeta}=\prod_v \frp_v^{m_v}$). 

By using \eqref{comp_height} and the bound mentioned above on the displacement of unipotent elements on horospheres of height one we finally get the following upper bound on $\ell(\zeta)$: 
\begin{equation}
\ell(\zeta) \ll \log D_F + \log |\fri|
\label{longueur}
\end{equation}
with an absolute constant.


\subsection{Flats}

There is a well--known bijection between the $\Gamma_F$-equivalence classes of quadratic forms over $\so_F$ and the top-dimensional compact flats in $M_F$ which we will now briefly describe; for details we refer to \cite[Section 4]{Sarnak} and \cite[Chapitre I.5]{Efrat} (which treat respectively the case of imaginary quadratic and totally real $F$, but similar arguments work in general). We will also relate the geometry of the flats to the invariants of the quadratic forms for later use. 

\subsubsection{Correspondance with quadratic forms}

The set of fundamental discriminants of $F$ is defined as follows:
$$
\sd_F = \{d\in\so_F:\: d\not\in\so_F^2,\, \exists x\in\so_F:\: d=x^2\pmod 4 \};
$$
it is the set of discriminants $d=b^2-4ac$ of quadratic forms $q = az_1^2+bz_1z_2+cz_2^2$ on $F^2$ when $a,b,c\in\so_F$ have no common factor in $\so_F$ (note that it may not be the case that the ideal $(a,b,c)$ is the whole of $\so_F$). To such a quadratic form $q$ we can associate an abelian subgroup of $\Gamma_F$ as follows:
$$
\Gamma_q = \SO(q)\cap\Gamma_F. 
$$
We will now construct a flat subspace $Z\subset X$ of maximal dimension, whose (setwise) stabilizer in $\Gamma_F$ is $\Gamma_q$ which will act cocompactly. For $v\in V_\infty$ let $\sigma_v$ be the corresponding embedding of $F$ in $\CC$ and define $Z_v$ to be:
\begin{itemize}
\item If $F_v\cong\CC$, the geodesic in $X_v\cong\HH^3$ between the isotropic lines in $\PP^1(F_v)=\pl X_v$ of the quadratic form $q^{\sigma_v}$; 
\item If $F_v\cong\RR$ and $d^{\sigma_v}>0$, the geodesic between the isotropic lines in $\PP^1(F_v)=\pl X_v$ of the quadratic form $q^{\sigma_v}$;
\item If $F_v\cong\RR$ and $d^{\sigma_v}<0$, the point in $X_v\subset\PP^1(F_v\otimes\CC)$ corresponding to an isotropic line of $q^{\sigma_v}$. 
\end{itemize}
Let 
$$
r_1'=|\{v\in V_\infty:\: F_v\cong\RR, d^\sigma>0\}, \quad m = r_2 + r_1' ; 
$$ 
note that $m$ is the dimension of maximal flats in $X$. There is a unique flat $Z_q$ in $X$ of dimension $m$ which contains all $Z_v$; clearly the setwise stabiliser $\stab_{\Gamma_F} Z_q = \Gamma_q$, and we will now explain (following \cite[Theorem 5.6]{Efrat}) that $\Gamma_q$ is of rank $m$, and thus acts cocompactly on $Z_q$ (since it acts properly discontinuously). 

Let $\fra$ be the ideal generated by $a,b,c$ in $\so_F$, let $E$ be the quadratic extension $F(\sqrt d)$ of $F$, and let $\so_q$ be the order:
$$
\so_q = \left\{\frac{t+u\sqrt d}2:\: t\in\so_F, u\in\fra^{-1} \right\}
$$
in $\so_E$. It is proven in \cite{Efrat} that the group $\Gamma_q$ consists of the matrices:
\begin{equation}
\Gamma_q = \left\{ \begin{pmatrix} \frac{t+bu}2 & cu \\ -au& \frac{t-bu}2 \end{pmatrix} :\: \eps=\frac{t+u\sqrt d}2\in\so_q, N_{E/F}(\eps)=1 \right\}
\label{param}
\end{equation}
where $N_{E/F}$ denotes the relative norm of the field extension $E/F$. A quick computation shows that $E$ has exactly $2r_1'$ real places and $2r_2+r_1-r_1'$ complex ones; it follows that the rank of the abelian group $\so_E^\times\cap\ker(N_{E/F})$ equals $r_2+r_1'=m$. 


\subsubsection{Volume and systole}

Let us define the notation:
$$
\sd_F^+ = \begin{cases}
              \{d\in\sd_F:\: \exists\sigma,\, d^\sigma >0\} & \text{ if } r_2=0; \\
              \sd_F & \text{ otherwise}.
          \end{cases}
$$
Thus the results of the previous paragraph mean that there is a bijection between maximal compact flats (of dimension $>0$) in $M_F$ and $\Gamma_F$-equivalence classes of integral quadratic forms on $F^2$ whose discriminant belongs to $\sd_F^+$. We will now describe the geometry of the flats in function of $d$. 

Recall that $m$ is the logarithmic Mahler measure defined in \ref{mahler}. As the matrices on the right-hand side of \eqref{param} has eigenvalues equal to $\eps^{\pm 1}$ it follows from \eqref{Mahler_length} that the lengths of closed geodesics in $\Gamma_q\bs Z_q$ are, up to a bounded factor, equal to $m(\eps)$ for $\eps\in\so_q^\times, N_{E/F}(\eps)=1$. 

As for the volume, it is proven in \cite[Chapitre I.7]{Efrat} that the volume of $\Gamma_q\bs Z_q$ is equal to the covolume of $\Log(\so_q^\times\cap\ker(N_{E/F}))$ in its $\RR$-span in $\RR^{V_\infty(E)}$; it is easily seen that $\Log\so_F^\times$ is orthogonal in $\RR^{V_\infty(E)}$ to $\ker(N_{E/F})$, and it follows that
\begin{align*}
\vol(\Gamma_q\bs Z_q) &= \vol(\so_q^\times\cap\ker(N_{E/F})) \\
                      &= \frac{[\so_E^\times:\so_q^\times]}{[N_{E/F}(\so_E^\times):N_{E/F}(\so_q^\times)]}\vol\ker(N_{E/F}) 
\end{align*}
and since we have 
$$
R_F = \vol\so_F^\times = 2^{-\frac{r_2+r_1'}2} [\so_F^\times:N_{E/F}(\so_E^\times)]^{-1} \: \frac{\vol\so_E^\times}{\vol\ker(N_{E/F})}
$$
we finally get that
\begin{equation}
\vol(\Gamma_q\bs Z_q) = 2^{-\frac{r_2+r_1'}2} \frac{R_E}{R_F} \times \frac{[\so_E^\times:\so_q^\times]}{[\so_F^\times:N_{E/F}(\so_q^\times)]}. 
\label{vol_flat}
\end{equation}
Note that the term $[\so_E^\times:\so_q^\times]$ on the right-hand side is bounded above by a constant depending only on $d,r$: indeed, it is bounded by a constant depending only on $p,N_{F/\QQ}(d)$ at each place of $F$ dividing $p$, and equal to 1 if $v$ does not divide $d$. 


\subsection{Singular locus}

We will be very brief here; since the eigenvalues of elliptic elements belong to quadratic extensions of $F$ their order is bounded by a constant depending only on $r$. As for the volume of their fixed flat we have the following (see \cite[equation (4.8)]{Sarnak}).

\begin{lem}
If $d\in\sd_F$ and $|d|_\infty>4$ then $\Gamma_q$ is torsion-free for all $q$ whose discriminant equals $d$. 
\label{vol_sing}
\end{lem}

\begin{proof}
If $\Gamma_q$ contains torsion this means that there exists a $\eps\in\so_q^\times, N_{E/F}(\eps)=1$ such that $|\eps|_v=1$ for all $v\in V_\infty(E)$; writing $\eps=(t+u\sqrt d)/2$ with $t^2-u^2d=4$ we get that $|t\pm u\sqrt d|_\infty=1$ and it follows that $|t|_\infty^2+|u|_\infty^2|d|_\infty=4$, hence $|d|_\infty\le 4$. 
\end{proof}


\section{Geometry of $M_F$ as $D_F\to+\infty$} \label{main}

\subsection{Counting flats in a cusp neighbourhood}

In this section we establish upper bounds for the number of flats containing geodesic of length less than a given $R$; in view of Proposition \ref{criter} they would be sufficient to prove than $M_F\xrightarrow[D_F\to+\infty]{\rm BS} X$. 

\subsubsection{A criterion for going through an horoball}

For $\xi=\frac\alpha\beta\in F$ recall that we have defined $\frb$ as the ideal $(\beta)/(\alpha,\beta)$. We recall that we denote by $B_\xi(Y)$ the horoball around $\xi$ of height $|\frb|^{-1}Y$; it follows from the description \eqref{calcy} of heights functions that it is equal to $\gamma_\xi^{-1}B_\infty(Y)$. 

\begin{lem}
Let $\xi\in F$, $a,b,c\in\so_F,\, q=ax_1^2+bx_1x_2+cx_2^2$ and $d=b^2-4ac$. Then the flat $Z_q$ goes through the horoball $B_\xi(Y)$ is and only if 
$$
\frac{|d|_\infty^{1/2}}{|a\xi^2+b\xi+c|_\infty} \ge 2^{r_1+r_2}Y. 
$$
If $\xi=\infty$ this condition degenerates to
$$
\frac{|d|_\infty^{1/2}}{|a|_\infty} \ge 2^{r_1+r_2}Y. 
$$
\label{through}
\end{lem}

\begin{proof}
It is simpler to prove this at infinity and then pass to a given $\xi$ than to try to do all computations for $\xi$; thus we will do so. Let $z_\pm=(-b\pm\sqrt d)/2a\in F(\sqrt d)$ be the roots of $az^2+bz+c$. If $v$ is a complex place or a real place such that $d^{\sigma_v}>0$ then the highest point in $Z_v$ with respect to $\infty$ is the point on the geodesic line between $z_\pm$ which is directly above $(z_++z_-)/2$, and the height of this point equals $|z_+-z_-|_v/2=\sqrt{|d|_v}/2|a|_v$. If $v$ is real and $d^{\sigma_v}<0$ then $Z_v$ is a point of height $\Ima(z_+)=\sqrt{|d|_v}/2|a|_v$. Thus the maximal height of a point on $Z_q$ equals $|d|_\infty^{1/2}/2^{r_1+r_2}|a|_\infty$, and this proves the lemma for $\xi=\infty$. 

Now let $\xi\in F$, and recall that we have defined
$$
\gamma_\xi = \begin{pmatrix} 0&1\\-1&\xi\end{pmatrix}. 
$$
Let $c$ be the flat of $X$ corresponding to the quadratic form $q=ax_1^2+bx_1x_2+cx_2^2$. Then the image of $c$ under $\gamma_\xi$ corresponds to the quadratic form $\gamma_\xi\cdot q:=q\circ\gamma_\xi^{-1}$ and we can compute that: 
\begin{align*}
\gamma_\xi\cdot q &= a(\xi x_1-x_2)^2 +b (\xi x_1 -x_2)x_1 + cx_1^2 \\
                 &= (a\xi^2+b\xi+c)x_1^2 - (b + 2a\xi)x_1x_2 + ax_2^2 .
\end{align*}
Thus the lemma follows from the computation we did just above. 
\end{proof}


\subsubsection{} 
We will need the following classical counting result: if $\Lambda$ is a lattice in $\RR^r$ then we have
\begin{equation}
|\{ v\in\Lambda:\: |v|\le \rho\}| \ll \frac{\rho^r}{\vol\Lambda} + \tau(\Lambda)
\label{gauss}
\end{equation}
where the constant depends only on $r$. This follows rather immediately from Minkowski's second theorem and the fact that one can find a basis of $\Lambda$ whose vectors realize its successive minima (see \cite[Lemma 2.2]{torsion1} for a detailed proof in the two-dimensional case).

Let $N_{\xi,d}(Y)$ be the number of closed flats which intersect $B_\xi(Y)$ and correspond to forms of discriminant $d$. We will now prove:

\begin{lem}
For any $\eps>0$ we have
$$
N_{\xi,d}(Y) \ll Y^{-r-\eps} |\frb|^{2+\eps} D_F^{-\frac 1 2}
$$
with a constant depending only on $d,\eps$. 
\label{count}
\end{lem}

\begin{proof}
Let $\wdt S\subset\so_F^3$ be the set of triples $(a,b,c), b^2-4ac=d$ corresponding to closed geodesics meeting $B$, and let $\Phi_\xi$ be the self-map of $F^3$ defined by 
$$
\Phi_\xi(a,b,c) = (a\xi^2+b\xi+c,2a\xi+b,a); 
$$
in other words $\Phi_\xi$ corresponds to the map $q\mapsto \gamma_\xi\cdot q$ in the space of quadratic forms. Thus it preserves the quadratic form $b^2-4ac$, and it follows from Lemma \ref{through} above that
$$
\Phi_\xi(\wdt S) \subset \left\{(a,b,c)\in \frb^{-2}\times\frb^{-1}\times\so_F: \: b^2-4ac = d, \, a\le 2^{r_1+r_2} \frac{|d|_\infty^{1/2}}Y \right\}. 
$$
Since $\Phi_\xi$ intertwines the actions of $\Lambda_\xi$ and $\gamma_\xi\Lambda_\xi\gamma_\xi^{-1}$ (which we identify with $\frb^2$ according to \eqref{unip_xi}) and the latter is given by 
$$
\frb^2\times\left(\frb^{-2}\times\frb^{-1}\times\so_F\right)\ni (u,(a,b,c)) \mapsto (a,b+2ua,c+ub+u^2a)
$$
we get that the quotient $S$ of $\wdt S$ by $\Gamma_\xi$ is identified through $\Phi_\xi$ with a subset of 
$$
S' = \left\{(a,b)\in\frb^{-2}\times\frb^{-1}:\: |a|\le 2^{r_1+r_2} \frac{|d|_\infty^{1/2}}Y, \exists c\in\so_D:\: b^2-4ac=d \right\}/\frb^2
$$
where the action of $\frb^2$ is given by $u\cdot(a,b)=(a,b+2ua)$. Now we fix an $\eps>0$; it is easily seen that for a given $d\in\so_F$ the number of solutions to the congruence $b^2=d\pmod{2a}$ has $\ll|a|^\eps$ solutions in $\frb^{-1}$ modulo $2a\frb^2$, where the constant is absolute: indeed, at each finite place there are at most two solutions, so the total number is bounded by $2^l$ (where $l$ is the number of prime divisors of $\frb$) which is $\ll|\frb|^\eps$ for all $\eps>0$. So we get that for a given $a\in\frb^{-2}$ there are $\ll|\frb^{-1}/a\frb|^\eps$ possible choices for $b\in\frb^{-1}\pmod{2a\frb}$ such that $(a,b)\in S'$. It follows that
\begin{align*}
|S'| &= \sum_{\substack{a\in\frb^{-2}\\ |a|_\infty\le \sqrt d /Y}} |\left\{b\in\frb^{-1}/2a\frb^2: b^2=d\pmod{2a\frb^2}\right\}|
     \ll \sum_{\substack{a\in\frb^{-2}\\ |a|_\infty\le \sqrt d /Y}} [\frb^{-1}:2a\frb^2]^\eps \\
     &\le 4  \sum_{\substack{a\in\frb^{-2}\\ |a|_\infty\le \sqrt d /Y}} |a\frb|^\eps \le  |\frb|^\eps(d/Y^2)^\eps \left| \left\{ a\in\frb^{-2}:\: |a|_\infty\le \frac{\sqrt d} Y \right\} \right| \\
     &\ll |\frb|^\eps(d/Y^2)^\eps \left( \frac {|d|_\infty^{r/2}}{Y^r\vol(\frb^{-2})}  + \tau(\frb) \right).
\end{align*}
(using \eqref{gauss} for $\Lambda=\frb$ at the third line) for any ideal $\fri$ in $\so_F$ the volume of its inverse $\fri^{-1}$ is equal to $|\fri|^{-1}\vol\so_F$ so the right-hand side is bounded above by
$$
Y^{-r-2\eps}|\frb|^{2+2\eps}D^{-\frac 1 2} + Y^{-2\eps}|\frb|^{2\eps} D^{\frac 1 2} 
$$
which finishes the proof of the lemma. 
\end{proof}


\subsection{Estimating the volume of the thin part of $M_F$}

Let $\frc_1, \ldots, \frc_{h_F}$ associated to $\xi_j$ are the integral representatives of lowest norm for the ideal classes of $F$ and $\xi_j = \alpha_j/\beta_j,j=1,\ldots,h_F$ such that $\frc_j=(\beta_j)/(\alpha_j, \beta_j)$. Let $\gamma_j=\Gamma_{\xi_j}$; by \eqref{reduction2} we know that the projections to $M_F$ of the horoballs $\gamma_j B_\infty(Y_j)$ cover it entirely when $Y_j = c_0|\frc_j| D_F^{-1/2}$ for a constant $c_0$ depending only on $r$. For a field $F$ let $\sd_{R, F}$ be the set of fundamental discriminants for $F$ which yield units in the set $T_R$ in the proof of Lemma \ref{nb_fini}. We put $N_\xi(Y) = \sum_{d\in \sd_{R,F}} N_{\xi, d}(Y)$; it follows from Lemma \ref{count} that for all $\eps>0$ we have:
\begin{align*}
\sum_{T,\alpha_1(T)\le R} \vol T & \le \sum_{j=1}^{h_F} N_{\xi_j}(Y_j)\max_{d\in\sd_{R,F}} \left(C_d\cdot \frac{R_{F(\sqrt d)}}{R_F} \right) \\
       & \ll \max_{d\in\sd_{R,F}} \left(C_d\cdot \frac{R_{F(\sqrt d)}}{R_F}\right) \sum_{j=1}^{h_F} \left( \frac{|\frc_j|}{D_F^{1/2}} \right)^{-r-\eps} \cdot |\frc_j|^{2+\eps} \cdot D_F^{-1/2} \\
       &\ll \max_{d\in\sd_{R,F}}\left( C_d\frac{R_{F(\sqrt d)}}{R_F} \right) \cdot D_F^{\frac{r-1}2+\eps} \sum_{j=1}^{h_F} |\frc_j|^{-r+2+\eps}
\end{align*}
Here $C_d$ is a constant independant of $F$, since the term $[\so_E^\times:\so_q^\times]$ on the right-hand side of \eqref{vol_flat} is bounded independantly of $F$. In particular, when $r=2$ (which is the case of interest for later) we get using the estimate $h_F\ll D_F^{1/2}\log D$ that:
\begin{equation}
\sum_{T,\alpha_1(T)\le R} \vol T \ll D_F^{1+\eps} \max_{d\in\sd_{R,F}} \left( C_d \cdot \frac{R_{F(\sqrt d)}}{R_F} \right).
\label{vol_hyp_fin}
\end{equation}
For cubic fields we record independently the following estimates for the number of compact flats. 

\begin{lem}\label{cubic}
For all $R, \eps > 0$ there exists $C_{r,\eps} > 0$ such that for all cubic fields $F$ we have $N_R(M_F)\le C_{R, \eps} D_F^{5/4+2\eps}$, where $N_R(M)$ is the number of maximal compact flats in $M$ whose systole is less than $R$. 
\end{lem}

\begin{proof}
From Lemma \ref{count} and \eqref{reduction2} we get the estimate:
$$
N_R(M_F) \ll_{\eps,R} D_F^{1+\eps}\sum_{j=1}^{h_F} |\frc_j|^{-1/2+\eps}. 
$$
The number of ideals $\frc_j$ of norm $\le D_F^a$ is less than $C_\eps D_F^{a + \eps}$ for some absolute $C_\eps>0$ (this follows from elementary arguments, see for example the proof of Theorem 3.5 in \cite{Lenstra}) and we get that for any $a>0$ we have:
$$
N_R(M_F) \ll_{\eps,R} D_F^{1+2\eps-a}|h_F| + D_F^{1+2\eps}D_F^{a+\eps}.
$$
Taking $a=1/4$ we get the estimate in the lemma. 
\end{proof}

Now we deal with the cusps: for a given $\xi$, the $S[u]$ which contribute to the $R$-thin part are those with $\log|\fri|\le CR$ according to \eqref{sys_Levi} (where $C$ depends only on $r$). Thus there is a uniform bound on their number, the volume of each of the associated compact flats is bounded above by $R_F$, and we have $\ell(u)\ll \log D_F$ by \eqref{longueur}. In particular, since $h_FR_F\ll D_F^{\frac 1 2}(\log D_F)^{r-1}$ we get that
\begin{equation}
\sum_{\xi,u} \vol \ell(u)S[u] \ll C_R h_F\log(D_F)R_F \ll D_F^{1/2}(\log D_F)^r.
\label{vol_cusphyp_fin}
\end{equation}
On the other hand, by \eqref{shape} and the description of the Levi component of cusps we have that the contribution of the unipotent elements is bounded by 
\begin{equation}
\sum_\xi R_F\tau(\Lambda_\xi) \ll h_FR_FD_F^{1/2} \ll D_F(\log D_F)^{r-1} 
\label{vol_unip_fin}
\end{equation}

Finally, the description of the compact singular locus yields that there exists a $R_0$ such that for $r=2$ we have
\begin{equation}
\sum_{j=1}^e \vol M_n[\gamma_j] \ll \sum_{T,\alpha_1(T)\le R_0} \vol T \ll D_F^{1+\eps} \max_{d\in\sd_{R_0}} \left( C_d\cdot \frac{R_{F(\sqrt d)}}{R_F} \right).
\label{vol_ell_fin}
\end{equation}


\subsection{Bianchi orbifolds}

We are interested here in the sequence $M_D=M_F, F=\QQ(\sqrt{-m}), D=D_F$ as the square-free, positive integer $m$ tends to infinity. In this case we have that $\so_F^\times$ is trivial as soon as $D>6$. The proof of Lemma \ref{nb_fini} yields that for any $R$, for $D$ large enough the only discriminants $d\in\sd_F$ which yield closed geodesics of length less than $R$ in $M_D$ are among a finite set $d_1,\ldots,d_k\in\ZZ$, and for each $d$ among these the unit group $\so_{F(\sqrt d)}^\times$ is equal to $\so_{\QQ(\sqrt d)}^\times$. It follows that $\max_{d\in\sd_R} C_d R_{F(\sqrt d})$ is bounded independantly of $F$ and thus \eqref{vol_hyp_fin} yields that for all $\eps>0$ we have
$$
\vol (M_D)_{\le R} \ll D^{1+\eps}
$$
(there is no Levi part in the cusps, and the contribution of unipotents is $\ll D\log D$  by \eqref{vol_unip_fin}). This finishes the proof of Theorem \ref{conv_Bianchi}.


\subsection{Hilbert--Blumenthal surfaces}

Here we consider the sequence of fields $F=\QQ(\sqrt D)$ for square--free $D>0$, as above we denote $M_F=M_D$. This case is more subtle, as we have that $\so_F^\times$ is of rank one, and $\so_E^\times$ is of rank three for any totally real extension of $F$, and we will have to use some elementary manipulations to circumvent this problem. In the sequel we arbitrarily choose an embedding $F \to \RR$ for all real fields we encounter and denote by $| \cdot |$ the absolute value we get this way.  

Recall that for any number field $L$ of degree $r$ we have the bound
\begin{equation}
R_L \ll D_L^{1/2}(\log D_L)^{r-1}.
\label{bound_reg}
\end{equation}
Now let $D,d\in\sd_\QQ^+$ be two square-free positive integers, $F=\QQ(\sqrt D),\, F'=\QQ(\sqrt{dD})$ and $E=F(\sqrt d)=F'(\sqrt d)$. Let $\eps,\eps'$ and $\eps_d$ be fundamental units (such that $|\eps|, |\eps'| > 1$) for the fields $F,F'$ and $\QQ(\sqrt d)$ respectively, so that we have $R_F=\sqrt 2\log|\eps|$ and $R_{F'}=\sqrt 2\log|\eps'|$. We number the real embeddings of $E$ as follows: $\sigma_1$ is the identity, and 
\begin{itemize}
\item $\sigma_2(\sqrt d)=\sqrt d$, $\sigma_2(\sqrt D)=-\sqrt D$;
\item $\sigma_3(\sqrt d)=-\sqrt d$, $\sigma_2(\sqrt D)=\sqrt D$;
\item $\sigma_4(\sqrt d)=-\sqrt d$, $\sigma_2(\sqrt D)=-\sqrt D$. 
\end{itemize}
Then we have $\eps^{\sigma_i}=\eps$ for $i=1,3$ and $\eps^{-1}$ for $i=2,4$, $(\eps')^{\sigma_i}=\eps'$ for $i=1,2$ and $(\eps')^{-1}$ for $i=3,4$ and finally $\eps_d^{\sigma_i}=\eps_d$ for $i=1,4$ and $\eps_d^{-1}$ for $i=2,3$. It follows that the images of $\eps,\eps'$ and $\eps_d$ by $\Log$ are pairwise orthogonal and thus they generate a sublattice of finite index in $\Log\so_E^\times$. This implies that
\begin{equation}
R_E \le C(d) R_F R_{F'} 
\label{mirror}
\end{equation}
where $C(d)\ll d^{1/2}$ by \eqref{bound_reg} for $L=\QQ(\sqrt d)$.

Now there are two possibilities: 
\begin{itemize}
\item[(i)] both $R_F\le D_F^{1/5}$ and $R_{F'}\le D_{F'}^{1/5}$; 
\item[(ii)] $R_{F'}\ge D_{F'}^{1/5}$. 
\end{itemize}
If we are in case (i) then we get from \eqref{mirror} that 
\begin{equation}
R_E/R_F \ll_d D_F^{1/5}. 
\label{case_1}
\end{equation}
On the other hand, if we are in case (ii) we get in that for all $d'\in\sd_\QQ^+$, $E'=F'(\sqrt{d'})$ we have 
$$
R_{E'} \ll D_{E'}^{1/2}(\log D_{E'})^3 \ll D_{F'}^{1/2} (\log D_{F'})^3
$$
so that we finally obtain
\begin{equation}
R_{E'}/R_{F'} \ll D_{F'}^{3/10}(\log D_{F'})^3 \ll D_{F'}^{2/5}
\label{case_2}
\end{equation}
with a constant depending only on $d$. 

Now we will define a sequence $F_n$ as follows: let $R_n=n$; by Lemma \ref{nb_fini} there is a $D(n)\in\sd_\QQ^+$ such that for all $D\in\sd_\QQ^+, D>D(n)$ the only $d\in\sd_{\QQ(\sqrt D)}$ contibuting geodesics of length less than $R_n$ belong to a finite set $\sd_{R_n}\subset\ZZ_{>0}$. Now define:
\begin{itemize}
\item $F_n=\QQ(\sqrt{D(n)})$ if we are in situation (i) for all $d\in\sd_{R_n}$;
\item $F_n=\QQ(\sqrt{dD(n)})$ if we are in situation (ii) for some $d\in\sd_{R_n}$. 
\end{itemize}
Thus in any case we obtain by \eqref{case_1} or by \eqref{case_2} that for all $d\in\sd_{R_n}$, $E_n=F_n(\sqrt d)$ we have $R_{E_n}/R_{F_n} \ll D_{F_n}^{2/5}$ and in the proof of \eqref{vol_hyp_fin} we get 
\begin{equation} \label{ci-dessus}
(\vol M_{F_n})_{\le R_n} \ll C(R_n) (\vol M_{F_n})^{1-1/10}
\end{equation}
where $C(R_n) = \max_{d \in \sd_{R_n}} C(d)$: by increasing $D(n)$ we may suppose that $C(R_n) \ll D(n)^\eps$ for all $\eps > 0$, so that Theorem \ref{Hilbert_conv} follows from \eqref{ci-dessus}.


\section{Convergence of congruence and maximal lattices} \label{last}

\subsection{Description of maximal lattices}

\subsubsection{Unit groups of quaternion algebras}

We recall for the reader's convenience, and for fixing notation, the description of the maximal lattices in an arithmetic commensurability class in $G_\infty=\SL_2(\CC)$. For details we refer to the book \cite{MR}. Such classes are in correspondance with certain quaternion algebras. We restrict in what follows to lattices defined over quadratic or cubic fields, hence we will only need to consider quaternion algebras defined over quadratic or cubic fields. Moreover, in the cubic case we need only consider fields which have only one real place and algebras which ramify at that place. Fix a field $F$ and a quaternion algebra $A/F$ satisfying to the requirements above; then the classification of quaternion algebras implies that there exists a set $S$ of places of $F$ not containing its complex place, containing the real place if $F$ is cubic, and of even cardinality, such that $A$ is isomorphic over $F$ to the unique quaternion algebra whose ramification locus is exactly $S$. 

Suppose we are in this setting; we let $\G=\SL_1(A)$ be the $F$-algebraic group defined by the units of norm 1 in $A$, and at each place $v\not\in S$ (including the complex place) we fix an isomorphism $I_v:\G(F_v)\cong\SL_2(F_v)$ and we put\footnote{This is equivalent to choosing a vertex of the Bruhat--Tits tree $X_v$ and an adjacent edge in $X_v$, and taking $K_v,K_{0,v}$ to be the stabilizers in the action of $\G(F_v)$ on $X_v$ of this vertex and edge respectively (see for example \cite[II.1.3]{Serre_arbres}).}
$$
K_v = I_v^{-1}\SL_2(\so_v), \quad K_{0,v} = I_v^{-1}\left\{ \begin{pmatrix}a&b\\c&d\end{pmatrix}\in\SL_2(\so_v):\: |c|_v < 1 \right\}.
$$ 
For each finite set of places $T\subset V_f-S$ we define a compact-open subgroup of $\G(\Ade_f)$ as follows:
$$
K_S(T) = \prod_{v\in S\cap V_f} \G(F_v) \times\prod_{v\in T} K_{0,v} \times\prod_{v\in V_f-S\cup T} K_v.
$$
We finally define:
$$
\Gamma_S(T) = \G(F)\cap K_S(T). 
$$
Then all maximal arithmetic subgroups of $\G(F)$ are locally conjugated to a $\Gamma_S(\emptyset)$ for some collection of isomorphisms $I_v$ as above. The conjugation is in general not global (see \cite[Chapter 6.7]{MR}) but for our purposes (counting closed geodesics) all lattices obtained this way are equivalent (see Lemma \ref{spec_ind} below) and we will not introduce further notation to distinguish between them. We recall that the lattice $\Gamma_S(T)$ is nonuniform if and only if $S=\emptyset$. 

Finally, we remark that the usual definition of the lattices $\Gamma_S(T)$ uses maximal orders and Eichler orders rather than the isomorphisms $I_v$: this is in particular the case in \cite{MR}. It is however trivial to pass from one to the other: the maximal order corresponding to $(I_v)$ is the one which is normalized by $K_v$ at each $v$. 


\subsubsection{Maximal lattices}

The lattices $\Gamma_S(\emptyset)$ are usually not maximal in $\SL_2(\CC)$, and not every maximal lattice in their commensurability class contains them. However any maximal lattice in $\SL_2(\CC)$ is obtained as the normalizer of some $\Gamma_S(T)$. We will denote this normalizer by
$$
\ovl\Gamma_S(T) = N_{\SL_2(\CC)} \Gamma_S(T).
$$
We have a very good control over the index of $\Gamma_S(T)$ in $\ovl\Gamma_S(T)$:

\begin{lem} \label{indice}
We have for all $F,S,T$ the equality
\begin{equation} \label{indeq}
[\ovl\Gamma_S(T):\Gamma_S(T)]  \le 2^{2 + |T| + |S|}h_F^{(2)}
\end{equation}
where $h_F^{(2)}$ is the order of the $2$-torsion subgroup of the class group of $F$. 
\end{lem}

\begin{proof}
  By \cite[Theorem 11.5.1]{MR} we have 
  \[
  \frac {[\ovl \Gamma_S(\emptyset) : \Gamma_S(T)]} {[\ovl \Gamma_S(T) : \Gamma_S(T)]} = 2^{-m} \prod_{\frp \in T} \left( |\frp| + 1 \right) 
  \]
  for some $0 \le m \le |T|$. It follows that
  \[
  [\ovl\Gamma_S(T):\Gamma_S(T)] \le 2^{|T|} \frac{[\ovl \Gamma_S(\emptyset) : \Gamma_S(T)]}{\prod_{\frp \in T} \left( |\frp| + 1 \right)} = 2^{|T|} [\ovl \Gamma_S(\emptyset) : \Gamma_S(\emptyset)]
  \]
  since $[\Gamma_S(\emptyset) : \Gamma_S(T)] = \prod_{\frp \in T} \left( |\frp| + 1 \right)$.

  The result then follows from \cite[Corollary 11.6.4, Theorem 11.6.5]{MR}. More precisely, there is a subgroup $\Gamma$ between $\ovl \Gamma_S(\emptyset)$ and $\Gamma_S(\emptyset)$ such that $[\Gamma : \Gamma_S(\emptyset)] \le 2^{|S|+2}$ \cite[Corollary 11.6.4]{MR} and $\ovl \Gamma_S(\emptyset)/\Gamma$ embeds in the 2-torsion subgroup of the class-group of $F$ \cite[Theorem 11.6.5]{MR}. 
\end{proof}


\subsubsection{Fixed points}

The following result is similar to \cite[Theorem 1.11]{7S} (but much, much simpler to prove). For any finite place $v$ of a number field we denote by $\frp_v$ the prime ideal $\{a\in\so_F,v(a)<1\}$ and by $q_v$ the cardinality of the residue field $f_v = \so_F/\frp_v$. 

\begin{lem}
There is a $\delta>0$ such that the for any $R>0$, any imaginary quadratic field $F$, any finite set $S\subset V_f$ and any loxodromic $\gamma\in\Gamma=\Gamma_S(\emptyset)$ whose minimal displacement is $\le R$ and any $T\subset V_f-S$ we have
$$
\left| \fix_{\Gamma/\Gamma_S(T)} \right| \le C 2^{|T|}
$$
with a constant $C$ depending only on $R$.
\label{fix}
\end{lem}

\begin{proof}
We will use the notation $K_v(m)$ for $v\in V_f, m\in\NN$ to denote the subgroup of matrices  in $K_v$ congruent to the identity modulo $\frp_v^{m}\so_v$. First we oserve that if $q_v^m>2\cosh(R/2)+2$ and $\gamma$ is as in the statement then we must have $\gamma\not\in K_v(m)$. To see this, just note that if $\gamma\in K_v(m)$ is not unipotent then we have $|\tr(\gamma)-2|\ge q_v^m$; in particular we get that $q_v^m\le 2\cosh(R/2)+2$, which proves the claim. 

Now we put $C=\prod_{v: q_v\le 2\cosh(R/2)+2} (q_v+1)$. For any $T\subset V_f-S$ we let $T_0=\{v\in T, q_v\le 2\cosh(R/2)+2\}$ and $T_1=T-T_0$. Let $v\in T_1$; the action of $\Gamma$ on $K_v/K_{0,v}$ is equivalent to the action on $\PP^1(f_v)$ via $\Gamma\to\PSL_2(f_v)$. So since $\gamma$ does not belong to the kernel $\pm K_v(1)$ of the latter action it acts non-trivially. Any non-trivial element of $\PSL_2(f_v)$ has at most two fixed directions, so we get that $\gamma$ has at most two fixed points on $K_v/K_{0,v}$. It follows that:
\begin{align*}
|\fix_{\Gamma/\Gamma_S(T)}(\gamma)| &= \prod_{v\in V_f-S} |\fix_{K_v/\ovl\Gamma_S(T)}(\gamma)| \\
       &= \prod_{v\in T} |\fix_{K_v/K_{0,v}}(\gamma)| \le 2^{|T_1|} \prod_{v\in T_0} |K_v/K_{0,v}| \\
       &\le C 2^{|T|}. 
\end{align*}
\end{proof}


\subsection{Length spectra}

\subsubsection{A complicated formula}

In the adelic proof of the Jacquet--Langlands correspondance one does not care about the length spectra since all congruence subgroups are considered at once. However, when one wants results linking the spectrum of explicit orbifolds it is necessary to be more precise about that. This is developed for Riemann surfaces in \cite{BJ}, the French-reading reader is also refered to Chapter 8 of the book \cite{spectre}; here we ``generalize'' their results to imaginary quadratic number fields using the following result which we quote\footnote{The quotation is not {\it verbatim}: her statement is in the language of orders, she considers also $S$-arithmetic groups and there are additional conditions in her statement which trivialize in our case.} from M.F. Vign\'eras' book \cite[Corollaire 5.17]{Vigneras}. We use the standard notation (same as in {\it loc. cit.}): for an order $\mB$ in a quadratic extension $E/F$, and a finite place $v$ of $F$ the Eichler symbol $\left(\frac{\mB}v\right)$ is given by:
\begin{itemize}
\item $1$ if $\mB_v:=\mB\so_v$ is not a maximal order in $E_v$ or if $v$ splits in $E$;
\item $-1$ if $\mB_v$ is maximal, and $v$ is inert in $E$;
\item $0$ if $\mB_v$ is maximal, and $v$ ramifies in $E$.
\end{itemize}
We also define an ad hoc Eichler symbol at infinity by:
\[
\eps_\infty(\lambda; A) = \begin{cases}
                        1 &\text{ if } F \text{ is quadratic or }E=F(\lambda) \text{ has no real place or } A \text{ splits at the real place of }F.\\
                        0 & \text{otherwise.}  
                      \end{cases}
\]
for $\lambda\in\ovl\ZZ$ of degree 2 over $F$. Let also $h(\mB)$ denote the class-number of $\mB$; the result is then:

\begin{theo}
Let $F$ be a quadratic or cubic field; $A$ a quaternion algebra over $F$. Let $E=F(\lambda)$ be a quadratic extension, where $\lambda\in\so_E$. Then the number of $\Gamma_S(T)$-conjugacy classes among the elements in $\Gamma_S(T)$ which have $\lambda$ as an eigenvalue is given by:
\[
\eps_\infty(\lambda;A)\cdot \sum_{\mB\ni\lambda} \frac{h(\mB)}{h_E} \cdot \prod_{v\in S} \left( 1-\left(\frac{\mB}{v}\right) \right) \cdot \prod_{w\in T} \left( 1+\left(\frac{\mB}{v}\right) \right)
\]
where the sum runs over all orders of $L$ which contain $\lambda$. 
\label{spec}
\end{theo}


\subsubsection{Comparison of length spectra}

First we have the obvious corollary of Theorem \ref{spec} above:

\begin{lem} \label{spec_ind}
The length spectrum of a group $\Gamma_S(\emptyset)$ does not depend on the collection $I_v$ (equivalently on the maximal order in $A_S$) chosen to define it. 
\end{lem}

The precise form of the formula in \ref{spec} can be used to compare between the cases where $S=\emptyset$ or not; the following lemma has a proof which follows exactly that of \cite[Lemma 3.4]{BJ} (see also \cite[Lemme 8.23]{spectre}), as the arguments used there are purely combinatorial.  

\begin{lem}
Let $F$ be a quadratic field, and for $\ell>0$ let $\mu(S,T;\ell)$ be the number of closed geodesics of length $\ell$ in $\Gamma_S(T)\bs\HH^3$. 
$$
\mu(S,T;\ell) = \sum_{S'\subset S} (-2)^{|S'|} \mu(\emptyset,S'\cup T;\ell).
$$
\label{spec_comp2}
\end{lem}

If $F$ is a cubic field we want to compare the length spectra of compact maximal congruence orbifolds defined over $F$ with that of the 5-dimensional orbifolds $\Gamma_\emptyset(T)\bs(\HH^3\times\HH^2)$ where $\Gamma_\emptyset=\SL_2(\so_F)$. Let $\mu^1(\emptyset,T';\ell)$ is the number of 1-dimensional maximal compact flats of length $\ell$ in $\Gamma_\emptyset(T')\bs(\HH^3\times\HH^2)$ (corresponding to elements which are elliptic at the real place). 

\begin{lem}
Let $F$ be a cubic field with one real place, then we have: 
$$
\mu(S,T;\ell) = \sum_{S'\subset S} (-2)^{|S'|} \mu^1(\emptyset,S'\cup T;\ell)
$$ 
where  
\label{spec_comp3}
\end{lem}


\subsection{Proof of Theorem \ref{Main1}}

Theorem \ref{Main1} follows from Proposition \ref{vol_thin} with input from Proposition \ref{closed_count} (taking into account the fact that in rank 1, $\vol(T) \le R$ for a $R$-thin flat) and \ref{fin_sing} in the compact case, and from the same plus \eqref{unip_max} in the noncompact case. 

\subsubsection{Covolume of maximal lattices}

We begin by a remark about the covolumes of maximal lattices. For a field $F$ and $S,T\subset V_f, \, S\cap T=\emptyset$ we denote by $M_S(T)$ the hyperbolic three-orbifold $\ovl\Gamma_S(T)\bs\HH^3$. The volume of $M_S(T)$ is given by the formula:
\begin{equation} \label{vol_max_lattice}
\vol M_S(T) \asymp D_F^{\frac 3 2} \prod_{P \in S} (|\mathfrak P| - 1) \frac{\prod_{\mathfrak P \in T} (|\mathfrak P| + 1)}{[\Gamma_S(T) : \ovl \Gamma_S(T)]} 
\end{equation}
with constants depending only on the degree of $F$ (this follows from \cite[(11.3) on page 333]{MR} and the formula for the index of $\Gamma_S(T)$ in $\Gamma_S(\emptyset)$). 

In case $F$ is quadratic, the 2-torsion subgroup of the class group is generated by the prime divisors of $D_F$ (see \cite[Corollary 1 to Theorem 39]{Froelich_Taylor}) and in particular $h_F^{(2)}$ is less than the number of rational prime factors of $D_F$ so we have by comparing with \eqref{vol_max_lattice}: 
\begin{equation} \label{small_index_quad}
[\ovl\Gamma_S(T):\Gamma_S(T)] \le 2^{|S| + |T| + 2}h_F^{(2)} \ll (\vol M_S(T))^\eps 
\end{equation}
for all $F,S,T$ and all $\eps>0$, where the constant depends only on $\eps$. In the case where $F$ is cubic and we have $h_F^{(2)} \ll D_F^{0.24}$ we get 
\begin{equation} \label{small_index_cubic}
[\ovl\Gamma_S(T):\Gamma_S(T)] \le 2^{|T| + |S| + 1} h_F^{(2)} \ll (\vol M_S(T))^{1/6 - \delta_0}
\end{equation}
for some $\delta_0>0$. 

\subsubsection{Closed geodesics}

Recall that for an hyperbolic three--orbifold we denote by $N_R(M)$ the number of closed geodesics of length $\le R$ in $M$.  

\begin{prop}
There is a $\delta>0$ such that for any $R>0$ there is a $C_R$ such that for all cubic or quadratic fields $F$ and all $S,T$ we have 
$$
N_R(M_S(T)) \le C_R(\vol M_S(T))^{1-\delta}. 
$$
\label{closed_count}
\end{prop}

\begin{proof}
Let $M_S'(T)=\Gamma_S(T)\bs\HH^3$; we will deduce below that for all $R>0$ we have
\begin{equation}
N_R(M_S'(T)) \le C_R'(\vol M_S'(T))^{5/6 + \eps}
\label{cc_cong}
\end{equation}
for any $\eps>0$. The proposition then follows easily from this and Lemma \ref{indice}, by the following argument: a theorem of Takeuchi \cite[Corollary 8.3.3]{MR} states that for every $\gamma\in\ovl\Gamma_S(T)$ we have $\gamma^2\in\Gamma_S(T)$, and it follows that $N_R(M_S(T))\le N_{2R}(M_S'(T))$. It follows from \eqref{cc_cong} and Lemma \ref{indice} that 
\[
N_R(M_S(T)) \le C_{2R}'\cdot 2^{|T|} h_F^{(2)} \cdot (\vol M_S(T))^{5/6 + \eps)}
\]
and since $2^{|T|} h_F^{(2)} \ll (\vol M_S(T))^{1/6-\delta_0}$ by \eqref{small_index_quad} (in this case we actually get a much better bound) or \eqref{small_index_cubic} we can conclude that $N_R(M_S(T)) \ll \vol(M_S(T))^{1-\delta}$ with $\delta=\delta_0 - \eps$, which can be made positive by taking $\eps$ small enough. 

The proof of \eqref{cc_cong} will follow from Theorem \ref{conv_Bianchi} by using Jacquet--Langlands and Lemma \ref{fix} with arguments similar to those used in \cite{7S},\cite{torsion2}; we will repeat them here for the reader's convenience. 

Fix $F,S$ and let $\gamma_1,\ldots,\gamma_r$ be representatives of the hyperbolic conjugacy classes in $\Gamma=\Gamma_S(\emptyset)$ corresponding to geodesics of length less than $R$ in $M_S(\emptyset)$: we include in this count the conjugacy classes which are nonprimitive. For $T\subset V_f-S$ and $i=1,\ldots,r$ let $n_i(T)$ be the number of $\Gamma_S(T)$-conjugacy classes (in $\Gamma_S(T)$) which are $\Gamma$-conjugated to $\gamma_i$; then we have 
$$
N_R(M_S'(T)) = \sum_{i=1}^r n_i(T). 
$$
Now $n_i(T)$ is equal to the number of fixed points of $\gamma_i$ in $\Gamma/\Gamma_S(T)$: indeed, for any subgroup $\Gamma'\subset\Gamma$ we see that $g\in\Gamma$ conjugates $\gamma_i$ into $\Gamma'$ if and only if $\gamma_i$ fixes $g\Gamma'\in\Gamma/\Gamma'$. Thus it follows from Lemma \ref{fix} that 
\begin{equation}
N_R(M_S'(T)) \le C2^{|T|} N_R(M_S'(\emptyset)) \ll_{R,\eps} (\vol M_S'(\emptyset))^\eps N_R(M_S'(\emptyset))
\label{niveau}
\end{equation}
since $2^{|T|}\ll (\vol M_S'(\emptyset))^\eps$ for all $\eps>0$. We now see that \eqref{cc_cong} can be deduced as follows:
\begin{itemize}
\item For $F$ quadratic and $S=T=\emptyset$ it follows from Theorem \ref{conv_Bianchi} (the case of Bianchi groups) and Lemma \ref{spec_ind} for the other maximal congruence lattices in the commensurability class;
\item For $F$ quadratic and $S=\emptyset$ it follows from the previous case and \eqref{niveau};
\item When $F$ is quadratic, $S\not=\emptyset$ we deduce \eqref{cc_cong} from the previous cases and Lemma \ref{spec_comp2}: we have
\begin{align*}
N_R(M_S'(T)) &\le \sum_{S'\subset S} 2^{|S-S'|} N_R(M_\emptyset'(S'\cup T)) \\  
            &\le 2^{|S|} \cdot|2^S|\cdot 2^{|S|}N_R(M_\emptyset'(T)) = 8^{|S|} N_R(M_\emptyset'(T))
\end{align*}
and $8^{|S|}\ll (\vol M_S'(T))^\eps$ for all $\eps>0$.
\item When $F$ is cubic the same sequence of arguments applies (for $S\not=\emptyset$) by using Lemma \ref{cubic} instead of Theorem \ref{conv_Bianchi} (note that $D_F^{5/4} \asymp (\vol M_F)^{5/6}$) and Lemma \ref{spec_comp3} instead of \ref{spec_comp2}. 
\end{itemize}
This completes the proof of the proposition. 
\end{proof}

Note that Propositions \ref{closed_count} and \ref{criter} suffice to imply the mere statement of BS-convergence in Theorem \ref{Main1}. 


\subsubsection{Unipotent elements}

Since we have $\tau(\fri)\le\tau(\so_F)\le D_F^{1/2}$ for all $F$ and ideals $\fri$ of $\so_F$, we have $\tau(\Lambda')\le D_F^{1/2}$ for any cuspidal subgroup $\Lambda'$ of any $\Gamma_\emptyset(T)$ (see \eqref{unip_xi}). Now if $\Lambda$ is a cuspidal subgroup of $\ovl\Gamma_\emptyset(T)$ and $\Lambda'$ the corresponding cuspidal subgroup of $\Gamma_\emptyset(T)$ we have $2\Lambda\subset\Lambda'$ and it follows that $\tau(\Lambda)\le 2\tau(\Lambda')$. On the other hand the number of cusps of $M_T'$ is less than $2^{|T|}h_F$ and it follows that the volume of the contribution of unipotent elements to the $R$-thin part of $M_\emptyset(T)$ is bounded above by 
\begin{equation}
C_R\cdot 2^{|T|}h_F\cdot 2D_F^{1/2} \ll D_F^{1+\eps}. 
\label{unip_max}
\end{equation}


\subsubsection{Singular locus}
\label{fin_sing}

We give here an estimate on the length of the singular locus of $M_S(T)$; we will be quite sketchy about this. First, the proof of Lemma \ref{vol_sing} applies in general and we thus have that there is a finite set $\sd_s\subset\ZZ$ of discriminants such that the singular geodesics in an arithmetic manifold in a commensurability class defined by an algebraic group over an imaginary quadratic field $F$ are associated to the quadratic extensions $F(\sqrt d),\, d\in\sd_s$. 

Now, the singular locus of the $M_S'(T)$ can be dealt with using Lemma \ref{fix}: we leave the details to the reader, one obtains that its length is $\le (\vol M_S'(T))^{1-\delta'}$ with the same $\delta'$ as for the number of geodesics. It remains to bound the length of those geodesics which are singular for the elements of order 2 in $\ovl\Gamma_S(T)$ not in $\Gamma_S(T)$. The key point here is that those geodesics correspond to tori whose reduction modulo the primes in $T$ belong to the reduction of $K_{0,v}$: this can be seen by a local argument at each $v\in T$; it follows that the length of the corresponding geodesics in $M_S'(T)$ is equal to the length of their projections in $M_S'(\emptyset)$, thus bounded, and the result follows from Proposition \ref{closed_count}.


\bibliographystyle{cdraifplain}
\bibliography{bib.bib}

\end{document}